\documentclass[10pt,reqno]{amsart}

\usepackage{fullpage,amsmath,amsthm,amsfonts,amssymb,enumitem,amsaddr}
\usepackage{aliascnt}
\usepackage{xcolor}
\usepackage[colorlinks=true,linkcolor=blue,citecolor=blue,urlcolor=blue,filecolor=blue]{hyperref}
\usepackage[capitalize,nameinlink,noabbrev]{cleveref}

\theoremstyle{plain}
\newtheorem{thm}{Theorem}[section]
\newaliascnt{lem}{thm}
\newtheorem{lem}[lem]{Lemma}
\aliascntresetthe{lem}
\newaliascnt{cor}{thm}
\newtheorem{cor}[cor]{Corollary}
\aliascntresetthe{cor}
\newaliascnt{prop}{thm}
\newtheorem{prop}[prop]{Proposition}
\aliascntresetthe{prop}

\theoremstyle{definition}
\newaliascnt{ex}{thm}
\newtheorem{ex}[ex]{Example}
\aliascntresetthe{ex}
\newaliascnt{rem}{thm}
\newtheorem{rem}[rem]{Remark}
\aliascntresetthe{rem}
\newaliascnt{defn}{thm}
\newtheorem{defn}[defn]{Definition}
\aliascntresetthe{defn}
\newaliascnt{prob}{thm}
\newtheorem{prob}[prob]{Problem}
\aliascntresetthe{prob}

\setlist[enumerate]{font=\upshape}

\crefname{thm}{Theorem}{Theorems}
\Crefname{thm}{Theorem}{Theorems}
\crefname{lem}{Lemma}{Lemmas}
\Crefname{lem}{Lemma}{Lemmas}
\crefname{cor}{Corollary}{Corollaries}
\Crefname{cor}{Corollary}{Corollaries}
\crefname{prop}{Proposition}{Propositions}
\Crefname{prop}{Proposition}{Propositions}
\crefname{ex}{Example}{Examples}
\Crefname{ex}{Example}{Examples}
\crefname{rem}{Remark}{Remarks}
\Crefname{rem}{Remark}{Remarks}
\crefname{defn}{Definition}{Definitions}
\Crefname{defn}{Definition}{Definitions}
\crefname{prob}{Problem}{Problems}
\Crefname{prob}{Problem}{Problems}
\crefname{section}{Section}{Sections}
\Crefname{section}{Section}{Sections}

\newcommand{\N}{\mathbb N}
\newcommand{\Z}{\mathbb Z}
\newcommand{\R}{\mathbb R}

\newcommand{\Hyp}{\mathbb H^2}
\newcommand{\diam}{\operatorname{diam}}
\newcommand{\Pack}{\operatorname{Pack}}
\newcommand{\Vol}{\operatorname{Vol}}
\DeclareMathOperator{\arsinh}{arsinh}
\DeclareMathOperator{\asdim}{asdim}
\newcommand{\filleq}{\simeq_{\mathrm{fill}}}

\title{Sequence distortion for metric spaces}

\author{Ilya Kapovich}

\address{Department of Mathematics and Statistics, Hunter College of CUNY\newline
  \indent 695 Park Ave, New York, NY 10065, U.S.A.\newline
  \indent ORCID 0000-0002-7694-6236}
\email{ik535@hunter.cuny.edu}



\dedicatory{Dedicated to the memory of my brother Michael Kapovich (1963-2026)}

\date{}

\hypersetup{
  pdftitle={Sequence distortion for metric spaces},
  pdfauthor={Ilya Kapovich},
  pdfkeywords={sequence distortion, snowflake metrics, quasi-isometry invariants, hyperbolic spaces, asymptotic cones}
}

\begin{document}

\begin{abstract}
We introduce \emph{sequence distortion spectrum}, a quasi-isometry invariant recording the large-scale distance profiles of sequences indexed by $\N$ or $\Z$ in a metric space. For a rate function $f:\N\to(0,\infty)$, extended by $f(t)=0$ for integers $t\le 0$, a sequence $(p_n)$ in a metric space $X$ is \emph{$f$-distorted} if there exists an integer $C\ge1$ such that for all $m,n$ we have
\[
      \frac{1}{C}f(\lfloor \frac{1}{C}|n-m|-C\rfloor)\le d(p_n,p_m)\le C f(C|n-m|+C)+C.
\]
This definition implies that $f(N)=O(N)$. For rate functions, realizability depends only on the ambient quasi-isometry type and the growth type of $f$. We classify the possible power rates $f(N)=N^\alpha$ (where $0<\alpha\le 1)$ for Euclidean spaces: in $\R$ only the linear rate $\alpha=1$ occurs, while in $\R^k$, $k\ge2$, the realizable exponents are exactly $1/k<\alpha\le1$. For a geodesic $\delta$-hyperbolic space $X$, no power rate $N^\alpha$ with $0<\alpha<1$ occurs. The hyperbolic plane also realizes the logarithmic rate. An exponential packing bound for $X$ rules out every $o(\log N)$ rate, but a proper CAT$(-1)$ surface of unbounded geometry realizes a log--log rate. In an arbitrary simplicial tree, every realizable rate is linear up to constants. Finally, we construct two pairs of proper geodesic spaces: the first has equivalent basepoint packing functions and the second equivalent uniform packing functions; both pairs have equal asymptotic dimensions and filling-function growth classes, and isometric asymptotic cones at the chosen wedge points for every common scaling sequence and ultrafilter. Yet sequence distortion distinguishes each pair, and the second pair has bounded geometry.
\end{abstract}

\maketitle


\section{Introduction}\label{s:introduction}

Large-scale geometry studies finitely generated groups up to quasi-isometry. Any two finite generating sets yield quasi-isometric word metrics, and the Milnor--\v{S}varc lemma identifies, at large scale, a group acting properly and cocompactly by isometries on a proper geodesic space with that space. Gromov's seminal monograph \emph{Asymptotic invariants of infinite groups} \cite{GromovAsymptotic} gave definitive form to the modern program of classifying finitely generated groups up to quasi-isometry and developing invariants that detect quasi-isometric rigidity. Standard general references include \cite{BH,Roe,deLaHarpe,DrutuKapovich,CornulierdeLaHarpe,BuyaloSchroeder}.

Such invariants draw on several neighboring areas. Growth, Dehn functions, higher-dimensional filling invariants, and Gromov hyperbolicity have geometric origins. For proper spaces and finitely generated groups, ends and proper-homotopy invariants describe topology at infinity; coarse homology and cohomology, asymptotic dimension, Higson-corona methods, and coarse duality provide explicitly large-scale analogues of algebraic-topological constructions. See \cite{RoeCoarseCohomology,Roe,Dranishnikov,Geoghegan,KapovichKleiner} for representative accounts. For hyperbolic groups and spaces, the boundary at infinity provides further invariants, including topological type, quasi-M\"obius or quasisymmetric structure, and conformal dimension; see, for example, \cite{GromovHyperbolic,GhysdeLaHarpe,KapovichBenakli,BuyaloSchroeder}. Asymptotic cones provide another source of large-scale information.

Several landmark results illustrate this power. Gromov proved that every finitely generated group of polynomial growth is virtually nilpotent \cite{GromovPolynomial}. Van den Dries and Wilkie later recast the argument in nonstandard terms \cite{vanDenDriesWilkie}, helping establish the modern ultrafilter formalism for asymptotic cones. Rigidity theorems of Schwartz for rank-one lattices and of Kleiner--Leeb for higher-rank symmetric spaces and Euclidean buildings show that sufficiently rich large-scale data can recover specific algebraic and geometric structure \cite{SchwartzRankOne,KleinerLeeb}.

Subgroup distortion is another source of nonlinear metric behavior; see \cite{GromovAsymptotic,Osin}. If $H$ is a finitely generated subgroup of a finitely generated group $G$, one compares intrinsic word length in $H$ with ambient word length in $G$. A standard example is the central cyclic subgroup in the discrete Heisenberg group
\[
        H_3(\Z)=\langle a,b,c\mid [a,b]=c,\ [a,c]=[b,c]=1\rangle.
\]
The identity $[a^n,b^n]=c^{n^2}$ and the decomposition $|N|=q^2+r$ where $0\le r<2q+1$ give $|c^N|_{H_3(\Z)}\lesssim \sqrt{|N|}$; the reverse inequality follows from the normal form or central-coordinate estimate. Hence the ambient length of $c^N$ is comparable to $\sqrt{|N|}$, while its intrinsic length in $\langle c\rangle$ is $|N|$. Yet distortion of a specified subgroup is not an invariant of the ambient group alone: a quasi-isometry need not respect multiplication and generally sends a subgroup neither to a subgroup nor to a set lying a bounded distance from one.

Sequence distortion retains the metric profile while discarding the algebraic requirement. Let $I$ be either $\N$ or $\Z$; we call these the semi-infinite and bi-infinite index sets, respectively. Given a \emph{rate function} $f:\N\to(0,\infty)$ (see \cref{d:rate-function}), extended to $\Z$ by setting $f(t)=0$ for $t\le0$, we ask whether $X$ contains a sequence $(p_n)_{n\in I}$ for which there exist constants $A,B>0$ and an integer $N_0\ge1$ such that
\begin{equation}\label{e:intro-rate}
        A f(|n-m|)\le d_X(p_n,p_m)\le B f(|n-m|)
        \qquad\text{whenever }|n-m|\ge N_0.
\end{equation}
We call such a sequence \emph{$f$-distorted}; see \cref{s:prelim-rates} for the precise definition and growth-type convention.  In \cref{p:distorted-normalizations} we obtain several equivalent characterizations of this notion, including formulations that allow additive constants and multiplicative rescaling of the argument. In particular, for a rate function $f$ and a sequence $(p_n)_n$, condition \eqref{e:intro-rate} is equivalent to the existence of an integer $C\ge 1$ such that for all $m,n\in I$ we have
\[
      \frac{1}{C}f(\lfloor \frac{1}{C}|n-m|-C\rfloor)\le d_X(p_n,p_m)\le C f(C|n-m|+C)+C.
\]

Unlike distortion or compression of an entire map between metric spaces, sequence distortion fixes the ambient space and asks only for an indexed sequence with the prescribed large-scale translation-invariant distance profile. The resulting \emph{sequence distortion spectrum}, defined in \cref{d:sequence-spectrum}, is a quasi-isometry invariant by \cref{p:spectrum-qi}.

The Heisenberg sequence $(c^n)$ realizes the square-root rate. Every realizable function is quasi-subadditive and at most linear; see \cref{l:rate-constraints}. Sequence distortion therefore records one-dimensional large-scale patterns within the natural range from bounded to linear growth. The diagnostic examples give two pairs: the first has equivalent basepoint packing functions and the second equivalent uniform packing functions; both pairs have equal asymptotic dimensions and filling-function growth classes, and isometric asymptotic cones at the chosen wedge points for every common scaling sequence and ultrafilter. Sequence distortion nevertheless separates each pair.

Two themes drive the proofs. A power rate $N^\alpha$ asks for a large-scale discretization of the snowflaked interval $([0,1],|s-t|^\alpha)$, so Euclidean quasiarc theory controls the critical exponent. A power-distorted sequence also produces a snowflaked interval in an asymptotic cone, while asymptotic cones of geodesic hyperbolic spaces are $\R$-trees. We recall the external inputs in \cref{s:prelim-background}.

The Euclidean classification is Theorem~\ref{mt:euclidean}, proved in \cref{t:euclidean}.

\begin{thm}[Euclidean spaces]\label{mt:euclidean}
Let $I\in\{\N,\Z\}$ and let $0<\alpha\le 1$.
\begin{enumerate}
\item The space $\R$ contains an $\alpha$-power-distorted sequence indexed by $I$ if and only if $\alpha=1$.
\item If $k\ge 2$, then $\R^k$ contains an $\alpha$-power-distorted sequence indexed by $I$ if and only if
\[
        \alpha>\frac1k.
\]
\end{enumerate}
\end{thm}

\noindent

For $N\ge 0$, put
\[
        \ell(N)=\log(1+N).
\]
The hyperbolic and tree results are collected in Theorem~\ref{mt:hyperbolic-tree} and proved in \cref{t:delta-power,p:hyperbolic-boundary-linear,c:hyp-power,c:log-sharp-hyp,t:simplicial-trees,c:tree-log}.

\begin{thm}[Hyperbolic spaces, the hyperbolic plane, and trees]\label{mt:hyperbolic-tree}
\leavevmode
\begin{enumerate}
\item Let $X$ be a geodesic $\delta$-hyperbolic metric space, let $I\in\{\N,\Z\}$, and let $0<\alpha\le 1$. If $X$ contains an $\alpha$-power-distorted sequence indexed by $I$, then $\alpha=1$.

\item Let $X$ be a geodesic $\delta$-hyperbolic metric space, and let $\partial X$ denote its sequential Gromov boundary.
\begin{enumerate}[label=\textup{(\alph*)}]
\item The space $X$ contains a $1$-power-distorted semi-infinite sequence if and only if $\partial X\ne\varnothing$.
\item The space $X$ contains a $1$-power-distorted bi-infinite sequence if and only if $|\partial X|\ge 2$.
\end{enumerate}

\item For each $I\in\{\N,\Z\}$, the sequence $p_n=(n,1)$, $n\in I$, in the upper half-plane model of $\Hyp$ satisfies
\[
        d_{\Hyp}(p_n,p_m)\asymp \ell(|n-m|)
        \qquad(m,n\in I,\ m\ne n).
\]

\item Let $I\in\{\N,\Z\}$, and let $f:\N\to(0,\infty)$ be nondecreasing and satisfy
\[
        f(N)=o(\ell(N)).
\]
Then $\Hyp$ contains no $f$-distorted sequence indexed by $I$.

\item Let $T$ be any simplicial tree with unit-length edges, let $I\in\{\N,\Z\}$, and let $f:\N\to(0,\infty)$ be nondecreasing and unbounded. If $(p_n)_{n\in I}$ is $f$-distorted in $T$, then there exist constants $c,L>0$ and $C\ge 0$ such that
\[
        c|n-m|-C\le d_T(p_n,p_m)\le L|n-m|
        \qquad(m,n\in I).
\]
Consequently,
\[
        f(N)\asymp N.
\]
\end{enumerate}
\end{thm}

\noindent

The packing statement in Theorem~\ref{mt:logarithmic} is \cref{l:exp-pack}; its group-theoretic consequence and CAT$(-1)$ example are \cref{c:fg-no-sublog} and \cref{ex:loglog}.
\begin{thm}[Logarithmic and sublogarithmic rates]\label{mt:logarithmic}
\leavevmode
\begin{enumerate}
\item Let $X$ satisfy an exponential packing bound, let $I\in\{\N,\Z\}$, and let $f:\N\to(0,\infty)$ be nondecreasing with $f(N)=o(\ell(N))$. Then $X$ contains no $f$-distorted sequence indexed by $I$.
\item Let $G$ be a finitely generated group equipped with a word metric, let $I\in\{\N,\Z\}$, and let $f:\N\to(0,\infty)$ be nondecreasing with $f(N)=o(\ell(N))$. Then $G$ contains no $f$-distorted sequence indexed by $I$.
\item There exists a proper geodesic CAT$(-1)$ surface $W$ such that, for each $I\in\{\N,\Z\}$, $W$ contains a sequence $(p_n)_{n\in I}$ satisfying
\[
        d_W(p_n,p_m)\asymp \log\log(e^e+|n-m|)
        \qquad(m,n\in I,\ m\ne n).
\]
\end{enumerate}
\end{thm}

\noindent

The following diagnostic examples are stated in Theorem~\ref{mt:diagnostic} and proved in \cref{s:qi-diagnostic}: part~(1) follows from \cref{p:seq-distinguishes,p:standard-invariants-agree}, and part~(2) from \cref{p:bounded-log-distinguishes,p:bounded-standard-invariants}.
\begin{thm}[Diagnostic examples]\label{mt:diagnostic}
Put
\[
        f_0(0)=0,
        \qquad
        f_0(N)=\log\log(e^e+N)\quad(N\ge 1).
\]
\begin{enumerate}
\item There exist $\delta\ge 0$, proper geodesic simply connected $\delta$-hyperbolic metric spaces $X,Y$, and basepoints $o_X\in X$, $o_Y\in Y$ such that the following statements hold.
\begin{enumerate}[label=\textup{(\alph*)}]
\item Neither $X$ nor $Y$ has bounded geometry.
\item The space $Y$ contains $f_0$-distorted sequences indexed by $\N$ and by $\Z$, whereas $X$ contains no $f_0$-distorted sequence indexed by $\N$.
\item If
\[
        V_X(R)=\Pack_1(B_X(o_X,R)),
        \qquad
        V_Y(R)=\Pack_1(B_Y(o_Y,R)),
\]
then
\[
        V_X(R)\asymp V_Y(R).
\]
\item One has
\[
        \asdim X=\asdim Y=2.
\]
\item If $\delta_X$ and $\delta_Y$ denote the Lipschitz filling functions, then
\[
        \delta_X\filleq L,
        \qquad
        \delta_Y\filleq L.
\]
\item For every scaling sequence $(r_i)$ with $r_i\to\infty$ and every non-principal ultrafilter $\omega$,
\[
        \operatorname{Cone}_\omega(X,o_X,(r_i))
        \cong
        \operatorname{Cone}_\omega(Y,o_Y,(r_i)).
\]
\end{enumerate}
Consequently, $X$ and $Y$ are not quasi-isometric.

\item Let $D\ge 5$ be an integer, let $T_D$ be the $D$-regular simplicial tree with unit-length edges, and define
\[
        X_b=T_D\vee\R^2,
        \qquad
        Y_b=T_D\vee\R^2\vee\Hyp,
\]
using chosen wedge points and the induced path metrics. Then $X_b$ and $Y_b$ are proper geodesic simply connected metric spaces of bounded geometry, and the following statements hold.
\begin{enumerate}[label=\textup{(\alph*)}]
\item Neither $X_b$ nor $Y_b$ is Gromov-hyperbolic.
\item The space $Y_b$ contains $\ell$-distorted sequences indexed by $\N$ and by $\Z$, whereas $X_b$ contains no $\ell$-distorted sequence indexed by $\N$.
\item If
\[
        V_Z^{\mathrm{unif}}(R)
        =\sup_{z\in Z}\Pack_1(B_Z(z,R))
        \qquad(Z\in\{X_b,Y_b\}),
\]
then
\[
        V_{X_b}^{\mathrm{unif}}(R)
        \asymp (D-1)^R
        \asymp V_{Y_b}^{\mathrm{unif}}(R).
\]
\item One has
\[
        \asdim X_b=\asdim Y_b=2.
\]
\item If $\delta_{X_b}$ and $\delta_{Y_b}$ denote the Lipschitz filling functions, then
\[
        \delta_{X_b}\filleq L^2,
        \qquad
        \delta_{Y_b}\filleq L^2.
\]
\item Let $o_{X_b}\in X_b$ and $o_{Y_b}\in Y_b$ denote the respective wedge points. For every scaling sequence $(r_i)$ with $r_i\to\infty$ and every non-principal ultrafilter $\omega$,
\[
        \operatorname{Cone}_\omega(X_b,o_{X_b},(r_i))
        \cong
        \operatorname{Cone}_\omega(Y_b,o_{Y_b},(r_i)).
\]
\end{enumerate}
Consequently, $X_b$ and $Y_b$ are not quasi-isometric.
\end{enumerate}
\end{thm}

\noindent
Part~(1) compares basepoint packing because neither space has bounded geometry; it does not compare finite uniform packing functions. Part~(2) compares uniform packing for bounded-geometry spaces. Both parts give exact equalities or growth equivalences for the listed invariants and compare cones from the displayed basepoints for every common scaling sequence and non-principal ultrafilter. They make no claim about cones with arbitrary moving basepoints. Non-quasi-isometry follows from the different realization behavior of $f_0$ in part~(1) and $\ell$ in part~(2), together with \cref{l:qi-invariance}.

\noindent

\medskip
\noindent\textbf{Organization.}
\Cref{s:prelim} fixes terminology and collects the packing and compactness tools used repeatedly. The model-space classifications are proved in \cref{s:euclidean,s:hyperbolic,s:simplicial-trees}. The diagnostic examples appear in \cref{s:qi-diagnostic}. \Cref{s:summary} gives comparison tables, and \cref{s:open-problems} records open problems and further questions.

\section{Preliminaries and definitions}\label{s:prelim}

This section fixes the terminology, conventions, and hypotheses used in the paper.

\subsection{Quasi-isometries, rate functions, and sequence distortion}\label{s:prelim-rates}

We use the convention $\N=\{1,2,3,\ldots\}$.

\begin{defn}[Quasi-isometric embeddings and quasi-isometries]\label{d:coarse-maps}
Let $X,Y$ be metric spaces and $F:X\to Y$ be a map.
\begin{enumerate}
\item For $L\ge 1$, $C\ge 0$, the map $F:X\to Y$ is an $(L,C)$-\emph{quasi-isometric embedding} if 
\[
        \frac{1}{L}d_X(x,x')-C\le d_Y(Fx,Fx')\le Ld_X(x,x')+C
\]
for all $x,x'\in X$.  The map $F$ is a \emph{quasi-isometric embedding} if $F$ is an $(L,C)$-quasi-isometric embedding for some $L\ge 1, C\ge 0$.
\item The map $F:X\to Y$ is a \emph{quasi-isometry} if $F$ is a quasi-isometric embedding and, in addition, there exists a constant $C'\ge 0$ such that for every $y\in Y$ there is $x\in X$ such that $d_Y(y, F(x))\le C'$. 
\end{enumerate}
\end{defn}

\begin{defn}[Comparison, dominance, and equivalence of functions]\label{d:function-comparison}

\begin{enumerate}[label=\textup{(\roman*)},leftmargin=*]
\item Let $\mathcal P\ne\varnothing$, and let $U,V\colon\mathcal P\to[0,\infty)$.  
\begin{itemize}
\item[(a)] We write
\[
        U\lesssim_{\mathcal P}V
\]
if there is a constant $C>0$ such that $U(p)\le C V(p)$ for every $p\in\mathcal P$.  
\item[(b)] We write
\[
        U\asymp_{\mathcal P}V
\]
if both $U\lesssim_{\mathcal P}V$ and $V\lesssim_{\mathcal P}U$; equivalently, there are constants $0<c\le C<\infty$ such that
\[
        cV(p)\le U(p)\le CV(p)
        \qquad(p\in\mathcal P).
\]
\item[(c)] When the parameter set has been explicitly specified by the quantifiers in the containing statement, we suppress the subscript $\mathcal P$ and write $U\lesssim V$ or $U\asymp V$.  The implicit constants are then independent of every variable ranging over $\mathcal P$, although they may depend on fixed ambient data and on parameters declared fixed in the containing statement.
\end{itemize}
\item Let $D$ be either $\N$ or $[0,\infty)$, and let $f,g:D\to[0,\infty)$.  

\begin{itemize}
\item[(a)] We write $f\lesssim_\infty g$ if there are $C>0$ and $N_0\ge0$ such that
\[
        f(N)\le Cg(N)
        \qquad(N\in D,\ N\ge N_0).
\]
\item[(b)] We write $f\asymp_\infty g$ if both $f\lesssim_\infty g$ and $g\lesssim_\infty f$, and then say that $f$ and $g$ are \emph{multiplicatively equivalent at infinity}.  We usually suppress the subscript $\infty$ and write $f\asymp g$; in particular, an unqualified assertion $f(N)\asymp g(N)$ for functions on $D$ means $f\asymp_\infty g$.

\item[(c)] Let $f,g:\N\to[0,\infty)$.  We write $f\preceq g$ and say that $f$ is \emph{dominated by $g$} if there are a constant $C\ge1$, an integer $M\ge1$, and $N_0\in\N$ such that
\[
        f(N)\le Cg(MN)+C
        \qquad(N\ge N_0).
\]
We write $f\succeq g$ if $g\preceq f$.  
\item[(d)] The functions $f$ and $g$ are \emph{growth equivalent}, written $f\simeq g$, if $f\preceq g$ and $g\preceq f$.  This is an equivalence relation; its classes are the \emph{growth types}.  
\end{itemize}
\end{enumerate}
None of the comparison or equivalence relations in this definition requires monotonicity.  Whenever monotonicity is needed, it is stated separately; in particular, it is part of the definition of a rate function below.
\end{defn}

\begin{defn}[Rate functions]\label{d:rate-function}
A \emph{rate function} is a function $f:\N\to(0,\infty)$ that is monotone nondecreasing, unbounded, and \emph{scaling invariant}: for every fixed integer $M\ge1$,
\[
        f(MN)\asymp f(N),
\]
where the comparison constants may depend on $M$.
By convention, we extend every rate function to $\Z$ by setting $f(t)=0$ for $t\le0$.
\end{defn}

The functions $N^\alpha$ for $\alpha>0$, $\log(1+N)$, and $\log\log(e^e+N)$ are rate functions.  Scaling invariance means that integer dilations of the input preserve the multiplicative equivalence class of the function.

\begin{lem}[Equivalence relations for rate functions]\label{l:rate-equivalences}
Let $f,g$ be rate functions.  Then
\[
        f\simeq g
        \quad\Longleftrightarrow\quad
        f\asymp g.
\]
\end{lem}

\begin{proof}
If $f\asymp g$, then $f\preceq g$ and $g\preceq f$ with input-dilation constant $M=1$, so $f\simeq g$.  Conversely, suppose first that $f\preceq g$.  Thus, for some $C\ge1$ and integer $M\ge1$,
\[
        f(N)\le Cg(MN)+C
\]
for all sufficiently large $N$.  Scaling invariance gives $g(MN)\le C_1g(N)$ for large $N$, and unboundedness absorbs the additive constant.  Hence $f\lesssim_\infty g$.  Applying the same argument to $g\preceq f$ proves $g\lesssim_\infty f$, and therefore $f\asymp g$.
\end{proof}

\begin{defn}[Distorted sequences]\label{d:distorted}
Let $(X,d_X)$ be a metric space, let $I\in\{\N,\Z\}$, and let $f:\N\to(0,\infty)$.  A sequence $(p_n)_{n\in I}$ in $X$ is \emph{$f$-distorted} if there are constants $A,B>0$ and an integer $N_0\ge1$ such that
\begin{equation}\label{e:fdist}
        A f(|n-m|)\le d_X(p_n,p_m)\le B f(|n-m|)
\end{equation}
whenever $m,n\in I$ and $|n-m|\ge N_0$.  We then say that $f$ is \emph{realizable in $X$ with index set $I$}.  For $f(k)=k^\alpha$, the sequence is also \emph{$\alpha$-power-distorted}.
\end{defn}

The constants and threshold in \eqref{e:fdist} may depend on $X$, its metric, $f$, and the sequence, but not on $m,n$.  In particular, repetitions $p_n=p_m$ are permitted, but only when $|n-m|<N_0$.

\begin{lem}[Bounded increments and a linear upper bound]\label{l:bounded-increments}
Let $(X,d_X)$ be a metric space, let $I\in\{\N,\Z\}$, let $f:\N\to(0,\infty)$, and suppose that $(p_n)_{n\in I}$ is $f$-distorted.  Then there is a constant $C>0$ such that
\[
        d_X(p_n,p_{n+1})\le C
\]
for every $n\in I$ with $n+1\in I$.  Consequently,
\begin{equation}\label{e:sequence-linear-upper}
        d_X(p_n,p_m)\le C|n-m|
        \qquad(m,n\in I).
\end{equation}
\end{lem}

\begin{proof}
Let $A,B,N_0$ be distortion data for $(p_n)$.  For every $n\in I$ with $n+1\in I$, both index gaps appearing below are at least $N_0$, and hence
\[
\begin{aligned}
        d_X(p_n,p_{n+1})
        &\le d_X(p_n,p_{n+N_0+1})
             +d_X(p_{n+N_0+1},p_{n+1})\\
        &\le Bf(N_0+1)+Bf(N_0).
\end{aligned}
\]
Thus the first assertion holds with $C=B\bigl(f(N_0+1)+f(N_0)\bigr)$.  The second follows by applying the triangle inequality along the consecutive terms from $p_n$ to $p_m$.
\end{proof}

Recall that, by convention, for a rate function $f:\N\to(0,\infty)$ we set $f(t)=0$ for every integer $t\le 0$.

\begin{prop}[Equivalent normalizations of distorted sequences]\label{p:distorted-normalizations}
Let $(X,d_X)$ be a metric space, let $I\in\{\N,\Z\}$, let $f:\N\to(0,\infty)$, and let $(p_n)_{n\in I}$ be a sequence in $X$.  Whenever $(p_n)$ is $f$-distorted, there is a constant $B_0>0$ such that
\begin{equation}\label{e:fdist-upper-all}
        d_X(p_n,p_m)\le B_0 f(|n-m|)
        \qquad(m,n\in I,\ m\ne n).
\end{equation}
If, in addition, $f$ is a rate function, then the following conditions are equivalent.
\begin{enumerate}[label=\textup{(\roman*)}]
\item The sequence $(p_n)$ is $f$-distorted.

\item There are constants $A,B>0$ and $C\ge0$ such that, for all $m,n\in I$,
\begin{equation}\label{e:fdist-additive}
        A f(|n-m|)-C
        \le d_X(p_n,p_m)
        \le B f(|n-m|)+C.
\end{equation}

\item There are a constant $C\ge1$ and integers $M\ge1$ and $R\ge0$ such that, for all $m,n\in I$, writing $k=|n-m|$,
\begin{equation}\label{e:fdist-coarse-affine}
        \frac1C f\left(\left\lfloor\frac{k-R}{M}\right\rfloor\right)-C
        \le d_X(p_n,p_m)
        \le C f(Mk+R)+C.
\end{equation}

\item There are an integer $K\ge1$ and constants $a,b>0$ such that the global upper estimate \eqref{e:fdist-upper-all} holds for some $B_0>0$, and the arithmetic subsequence $q_j=p_{Kj}$ satisfies
\begin{equation}\label{e:fdist-arithmetic-subsequence}
        a f(|i-j|)
        \le d_X(q_i,q_j)
        \le b f(|i-j|)
        \qquad(i,j\in I,\ i\ne j).
\end{equation}

\item There exists an integer $C\ge 1$ such that for all $m,n\in I$ we have
\begin{equation}\label{e:distorted-simplified}
      \frac{1}{C}f(\lfloor \frac{1}{C}|n-m|-C\rfloor)\le d_X(p_n,p_m)\le C f(C|n-m|+C)+C.
\end{equation}

\end{enumerate}
Whenever these equivalent conditions hold, every bounded subset $E\subseteq X$ contains only finitely many terms of the sequence, in the precise sense that
\[
        \{n\in I:p_n\in E\}
\]
is finite.  Moreover, if \eqref{e:fdist} holds with threshold $N_0$, then \eqref{e:fdist-arithmetic-subsequence} holds for every integer $K\ge N_0$, with constants that may depend on $K$.
\end{prop}

\begin{proof}
Suppose first, without assuming that $f$ is a rate function, that $(p_n)$ is $f$-distorted with data $A,B,N_0$.  Let $C>0$ be the constant from \cref{l:bounded-increments}.  If $m<n$ and $r=n-m<N_0$, then \eqref{e:sequence-linear-upper} gives
\[
        d_X(p_m,p_n)\le Cr.
\]
Together with the original upper estimate for gaps at least $N_0$, and using $f(r)>0$ for $1\le r<N_0$, these finitely many estimates can be absorbed into a constant $B_0>0$, proving \eqref{e:fdist-upper-all}.

Assume from now on that $f$ is a rate function.  Continuing under condition~\textup{(i)}, choose
\[
        C\ge A\max\{f(r):0\le r<N_0\}.
\]
Then the lower inequality in \eqref{e:fdist-additive} is automatic for $|n-m|<N_0$, while for larger gaps it follows from \eqref{e:fdist}; the global upper estimate proves the other inequality.  Thus (i) implies (ii).

Suppose that (ii) holds.  Since $f$ is unbounded, for all sufficiently large $k$ one has
\[
        C\le \frac{A}{2}f(k)
        \qquad\text{and}\qquad
        C\le f(k).
\]
Hence \eqref{e:fdist-additive} gives
\[
        \frac{A}{2}f(k)
        \le d_X(p_n,p_m)
        \le (B+1)f(k)
\]
whenever $k=|n-m|$ is sufficiently large.  Thus (ii) implies (i).

Condition (ii) implies (iii) by taking $M=1$, $R=0$, and increasing a single constant $C$ if necessary.  Conversely, suppose that (iii) holds.  Put
\[
        h(k)=\left\lfloor\frac{k-R}{M}\right\rfloor.
\]
For all sufficiently large $k$, one has $h(k)\ge1$ and $k\le 2M h(k)$.  By monotonicity and scaling invariance,
\[
        f(k)\lesssim f(h(k)).
\]
Also, for $k\ge1$,
\[
        Mk+R\le (M+R)k,
\]
so again monotonicity and scaling invariance give
\[
        f(Mk+R)\lesssim f(k).
\]
After increasing the threshold, unboundedness absorbs the additive constants in \eqref{e:fdist-coarse-affine}; hence (i) follows.

Condition~\textup{(v)} implies condition~\textup{(iii)} by taking $M=C$ and $R=C^2$ and using monotonicity in the upper estimate.  Conversely, suppose that~\textup{(i)} holds with data $A,B,N_0$, and let $B_0$ be the global upper constant in \eqref{e:fdist-upper-all}.  Choose an integer
\[
        C\ge \max\left\{1,B_0,A^{-1},\sqrt{N_0}\right\}.
\]
For $k=|n-m|<C^2$, the lower estimate in~\textup{(v)} is automatic because its argument of $f$ is negative.  For $k\ge C^2$, monotonicity gives
\[
        \frac1C f\left(\left\lfloor\frac{k}{C}-C\right\rfloor\right)
        \le A f(k)\le d_X(p_n,p_m).
\]
For $k\ge1$, the global upper estimate and monotonicity give
\[
        d_X(p_n,p_m)\le B_0f(k)\le C f(Ck+C)+C,
\]
and the case $k=0$ is trivial.  Thus~\textup{(i)} implies~\textup{(v)}, completing the comparison of these conditions.

It remains to compare (i) and (iv).  The global upper estimate has already been proved under (i).  Fix an integer $K\ge N_0$.  For $i\ne j$, the index gap $K|i-j|$ is at least $N_0$, and therefore
\[
        A f(K|i-j|)
        \le d_X(p_{Ki},p_{Kj})
        \le B f(K|i-j|).
\]
Scaling invariance, together with adjustment over finitely many initial values, yields \eqref{e:fdist-arithmetic-subsequence}.  Thus (i) implies (iv), as well as the final assertion concerning every $K\ge N_0$.

Conversely, suppose that (iv) holds.  Let $m<n$, put $k=n-m$, and choose
\[
        i=\left\lceil\frac{m}{K}\right\rceil,
        \qquad
        j=\left\lfloor\frac{n}{K}\right\rfloor.
\]
For all sufficiently large $k$, one has $i<j$ and
\[
        j-i\ge \frac{k}{2K}.
\]
The two endpoint gaps are less than $K$, so \eqref{e:fdist-upper-all} gives a uniform constant $D\ge0$ such that
\[
        d_X(p_m,p_{Ki})\le D,
        \qquad
        d_X(p_n,p_{Kj})\le D.
\]
The triangle inequality and \eqref{e:fdist-arithmetic-subsequence} therefore give
\[
        d_X(p_m,p_n)
        \ge a f(j-i)-2D.
\]
Since $k\le2K(j-i)$, monotonicity and scaling invariance imply $f(k)\lesssim f(j-i)$.  Unboundedness absorbs $2D$ for all sufficiently large $k$, giving the required lower estimate in \eqref{e:fdist}.  The required upper estimate is already contained in \eqref{e:fdist-upper-all}.  Hence (iv) implies (i).

Also, if a bounded subset $E\subseteq X$ contained $p_n$ for infinitely many indices $n$, those indices would contain pairs with arbitrarily large gaps, contradicting the lower estimate in (i).  This proves the remaining assertion.

\end{proof}

Thus, for a fixed sequence, the large-gap definition permits short-range repetitions or other degeneracies. It does not change which rate functions are realizable: every distorted sequence has a fixed arithmetic subsequence satisfying the all-gap multiplicative estimates.

\begin{rem}
A bi-infinite example always gives a semi-infinite one by restriction.  The converse need not hold, as the tree examples below show.
\end{rem}

\begin{lem}[Equivalent rates give equivalent distortion]\label{l:growth-representatives}
Let $f,g$ be rate functions with $f\simeq g$ (equivalently, $f\asymp g$), and let $(p_n)_{n\in I}$ be a sequence in a metric space $X$, where $I\in\{\N,\Z\}$.  Then $(p_n)$ is $f$-distorted if and only if it is $g$-distorted.
\end{lem}

\begin{proof}
By \cref{l:rate-equivalences}, there are constants $c,C>0$ and an integer $N_1\ge1$ such that
\[
        c g(N)\le f(N)\le Cg(N)
        \qquad(N\ge N_1).
\]
Increasing the distortion threshold to at least $N_1$ transfers the two-sided estimates between $f$ and $g$.
\end{proof}

\begin{prop}[Sequence distortion as a quasi-isometry invariant]\label{l:qi-invariance}
Let $I\in\{\N,\Z\}$, let $f$ be a rate function, and let $X,Y$ be quasi-isometric metric spaces.  Then $X$ contains an $f$-distorted sequence indexed by $I$ if and only if $Y$ does.
\end{prop}

\begin{proof}
We prove one implication.  Let $F:X\to Y$ be an $(L,C)$-quasi-isometry and suppose that $(p_n)$ is $f$-distorted in $X$.  By condition~\textup{(ii)} of \cref{p:distorted-normalizations}, there are constants $a,b>0$ and $D\ge0$ such that, for all $m,n\in I$,
\[
        a f(|n-m|)-D
        \le d_X(p_n,p_m)
        \le b f(|n-m|)+D.
\]
It follows that
\[
        \frac{a}{L}f(|n-m|)-\left(\frac{D}{L}+C\right)
        \le d_Y(F(p_n),F(p_m))
        \le Lb f(|n-m|)+(LD+C)
\]
for all $m,n\in I$.  Thus the image sequence again satisfies condition~\textup{(ii)} of \cref{p:distorted-normalizations}, and hence is $f$-distorted in $Y$.  Applying the same argument to a coarse inverse of $F$ proves the converse.
\end{proof}

\begin{defn}[Sequence distortion spectrum]\label{d:sequence-spectrum}
Let $X$ be a metric space and $I\in\{\N,\Z\}$.  The \emph{$I$-indexed sequence distortion spectrum} of $X$ is
\[
        \mathcal D_I(X)
        =\bigl\{[f]_{\simeq}: f\text{ is a rate function realizable in $X$ with index set $I$}\bigr\}.
\]
Here $[f]_{\simeq}$ denotes the growth type of $f$.  We also refer to $\mathcal D_{\N}(X)$ and $\mathcal D_{\Z}(X)$ as the \emph{one-sided} and \emph{two-sided sequence distortion spectra}, respectively.
\end{defn}

\begin{prop}[Quasi-isometry invariance of the sequence distortion spectrum]\label{p:spectrum-qi}
Let $X$ be a metric space and $I\in\{\N,\Z\}$.  If rate functions $f\simeq g$, then $f$ is realizable in $X$ with index set $I$ exactly when $g$ is; thus $\mathcal D_I(X)$ is well defined.  If $X$ and $Y$ are quasi-isometric, then
\[
        \mathcal D_I(X)=\mathcal D_I(Y).
\]
\end{prop}

\begin{proof}
The first assertion follows from \cref{l:growth-representatives}, and the second from \cref{l:qi-invariance}.
\end{proof}

\begin{rem}[At-most-linear rates]\label{r:at-most-linear-rates}
The bounded-increment estimate in \cref{l:bounded-increments}, combined with the lower inequality in \eqref{e:fdist}, shows directly that every realizable function satisfies $f(N)\lesssim N$.  Thus, in studying sequence distortion, it is enough to consider functions of at most linear growth.  Within this range, the scaling-invariance condition in \cref{d:rate-function} is a mild regularity requirement rather than a burdensome restriction: it is satisfied by the standard power rates $N^\alpha$ with $0<\alpha\le1$, as well as the logarithmic and iterated-logarithmic rates considered here, while excluding irregular behavior under fixed integer rescalings.  The next subsection records at-most-linearity and the related quasi-subadditivity constraint precisely.
\end{rem}

\subsection{Elementary restrictions on realizable functions}\label{s:prelim-rate-constraints}

\begin{lem}[Elementary constraints on realizable functions]\label{l:rate-constraints}
Let $I\in\{\N,\Z\}$, and suppose that $(p_n)_{n\in I}$ is $f$-distorted with data $A,B,N_0$. Then:
\begin{enumerate}[label=\textup{(\roman*)}]
\item if $p_n=p_m$, then $|n-m|<N_0$; in particular, every point of $X$ occurs at most $N_0$ times in the sequence;
\item there is a constant $C_0\ge1$ such that, for all integers $m,n\ge1$,
\begin{equation}\label{e:rate-quasi-subadditive}
        f(m+n)\le C_0\bigl(f(m)+f(n)\bigr);
\end{equation}
\item after increasing $C_0$ if necessary, for all $N\ge1$,
\begin{equation}\label{e:rate-at-most-linear}
        f(N)\le C_0N.
\end{equation}
In particular, every realizable function is at most linear.
\end{enumerate}
\end{lem}

\begin{proof}
The first assertion follows immediately from the lower bound in \eqref{e:fdist}; a set of integer indices having pairwise differences less than $N_0$ has at most $N_0$ elements.  Let $B_0$ be the global upper constant supplied by the first assertion of \cref{p:distorted-normalizations}.  If $m+n\ge N_0$, then the indices $1$, $1+m$, and $1+m+n$ and the triangle inequality give
\[
        A f(m+n)
        \le d_X(p_1,p_{1+m+n})
        \le d_X(p_1,p_{1+m})+d_X(p_{1+m},p_{1+m+n})
        \le B_0f(m)+B_0f(n).
\]
Increasing the constant to cover the finitely many pairs with $m+n<N_0$ proves \eqref{e:rate-quasi-subadditive}.  Let $C>0$ be the constant from \cref{l:bounded-increments}.  For $N\ge N_0$, the lower distortion estimate and \eqref{e:sequence-linear-upper} give
\[
        A f(N)
        \le d_X(p_1,p_{1+N})
        \le CN.
\]
Increasing the constant once more to cover $N<N_0$ proves \eqref{e:rate-at-most-linear}.
\end{proof}

In particular, no power function $N^\alpha$ with $\alpha>1$ is realizable.  Accordingly, all power-rate classifications below restrict to $0<\alpha\le1$.

\subsection{Hyperbolicity, trees, and asymptotic cones}\label{s:prelim-coarse}

A \emph{geodesic} in a metric space $X$ is an isometric embedding $\gamma:J\to X$, where $J$ is an interval in $\R$.  Its image is a \emph{geodesic segment}, \emph{geodesic ray}, or \emph{geodesic line} when $J$ is, respectively, a compact interval, a half-line, or all of $\R$.  We use the standard notions of a geodesic metric space and a proper metric space.

\begin{defn}[Hyperbolicity, trees, and cones]\label{d:hyp-cones}
A geodesic metric space is \emph{$\delta$-hyperbolic} if every geodesic triangle is $\delta$-thin.  For a basepoint $o\in X$, write
\[
        (x\mid y)_o=\frac12\bigl(d(x,o)+d(y,o)-d(x,y)\bigr).
\]
A sequence $(x_n)$ is a \emph{Gromov sequence} if $(x_m\mid x_n)_o\to\infty$ as $m,n\to\infty$; two Gromov sequences $(x_n)$ and $(y_n)$ are \emph{equivalent} if $(x_m\mid y_n)_o\to\infty$ as $m,n\to\infty$.  The set of equivalence classes is the \emph{sequential Gromov boundary} $\partial X$; for a hyperbolic space it is independent of the choice of $o$ up to canonical identification.

An \emph{$\R$-tree} is a geodesic metric space in which every two points are joined by a unique embedded arc, and this arc is a geodesic segment.  The \emph{valence} of a point in an $\R$-tree is the number of connected components of its complement.  A metric space is \emph{homogeneous} if its isometry group acts transitively on points.  An $\R$-tree has \emph{constant valence} if the valence is the same at every point.  A \emph{CAT$(-1)$ space} is a geodesic metric space whose geodesic triangles are no thicker than comparison triangles in the complete simply connected surface of constant curvature $-1$; in particular, by Cartan--Hadamard, a complete simply connected Riemannian manifold with sectional curvature at most $-1$ is CAT$(-1)$.

Given basepoints $x_i$, scaling factors $r_i\to\infty$, and a non-principal ultrafilter $\omega$, the \emph{asymptotic cone} $\operatorname{Cone}_\omega(X,(x_i),(r_i))$ is the ultralimit of $(X,r_i^{-1}d,x_i)$.  For background on hyperbolic and CAT spaces, $\R$-trees, and asymptotic cones, see \cite{BH,Chiswell,BuyaloSchroeder,DrutuKapovich}.  Throughout, $X\cong Y$ means that the metric spaces $X$ and $Y$ are isometric.
\end{defn}

\subsection{Dimension, filling, and metric regularity}\label{s:prelim-dimension-regularity}

\begin{defn}[Asymptotic dimension]\label{d:asdim}
The \emph{asymptotic dimension} of a metric space $X$ is at most $n$, written $\asdim X\le n$, if for every $R>0$ there is a uniformly bounded cover of $X$ whose \emph{$R$-multiplicity} is at most $n+1$, meaning that every ball of radius $R$ meets at most $n+1$ members of the cover.
\end{defn}

\begin{defn}[Filling invariants and functions]\label{d:filling}
\begin{enumerate}
\item For a simply connected geodesic space $X$ and a Lipschitz loop $\gamma:S^1\to X$, we use the standard parametrized Hausdorff area of a Lipschitz map $u:D^2\to X$,
\[
        \operatorname{Area}(u)
        =\int_{D^2}\mathbf J_2\bigl(\operatorname{ap\,md}u_z\bigr)\,d\mathcal L^2(z),
\]
where $\operatorname{ap\,md}u_z$ is the approximate metric derivative and $\mathbf J_2$ is its Hausdorff Jacobian; see \cite{LytchakWenger}.  This area is nonincreasing under postcomposition by a $1$-Lipschitz map.  For a smooth Riemannian target it agrees with the usual parametrized area and, for every smooth $2$-form $\omega$ of comass at most $1$, satisfies
\begin{equation}\label{e:area-form-bound}
        \left|\int_{D^2}u^*\omega\right|\le \operatorname{Area}(u).
\end{equation}
We shall also use the weak Stokes formula in the following standard form: if $M$ is a smooth Riemannian manifold, $u:D^2\to M$ is Lipschitz, and $\eta$ is a smooth $1$-form on a neighborhood of $u(D^2)$, then
\begin{equation}\label{e:lipschitz-stokes}
        \int_{D^2}u^*(d\eta)
        =\int_{\partial D^2}(u|_{\partial D^2})^*\eta.
\end{equation}
Let $\operatorname{FillArea}(\gamma)$ be the infimum of $\operatorname{Area}(u)$ over Lipschitz maps $u:D^2\to X$ with boundary trace $\gamma$; by convention, this infimum is $\infty$ if no such Lipschitz filling exists.  The \emph{Lipschitz filling function} is
\[
        \delta_X(L)=\sup\{\operatorname{FillArea}(\gamma):
        \operatorname{length}(\gamma)\le L\}.
\]
\item For nondecreasing functions $u,v:[0,\infty)\to[0,\infty)$, define \emph{filling dominance}, written $u\preceq_{\mathrm{fill}}v$, by requiring a constant $C\ge1$ such that
\[
        u(L)\le C v(CL+C)+CL+C
        \qquad(L\ge0).
\]
The functions are \emph{filling equivalent}, written $u\filleq v$, if $u\preceq_{\mathrm{fill}}v$ and $v\preceq_{\mathrm{fill}}u$.  We abbreviate equivalence to $L\mapsto L$ and $L\mapsto L^2$ by $u\filleq L$ and $u\filleq L^2$, respectively.  
\end{enumerate}
\end{defn}
Only these linear and quadratic filling classes occur in \cref{s:qi-diagnostic}.  Reshetnyak majorization \cite[Theorem~II.1A.6]{BH} applies to every closed rectifiable curve in a CAT$(\kappa)$ space, with the usual length restriction when $\kappa>0$; no simplicity assumption on the curve is required.  It gives a convex region in the model plane and a $1$-Lipschitz map of that region to the CAT$(\kappa)$ space whose boundary restriction preserves arclength and parametrizes the given curve.  The weakly monotone degree-one boundary reparametrizations arising from arclength parametrization do not change filling area: the two boundary maps are joined by a Lipschitz annulus whose image lies in the curve and therefore has zero parametrized area.  The Euclidean and hyperbolic isoperimetric inequalities therefore give
\[
        \operatorname{FillArea}(\gamma)\le \frac{L^2}{4\pi}
        \quad\text{in CAT$(0)$ spaces},
        \qquad
        \operatorname{FillArea}(\gamma)\le L
        \quad\text{in CAT$(-1)$ spaces}
\]
for every Lipschitz loop $\gamma$ of length $L$.  For background on asymptotic dimension, see \cite{GromovAsymptotic,Dranishnikov,Roe}; for filling invariants and parametrized Hausdorff area, see \cite{GromovAsymptotic,BH,DrutuKapovich,LytchakWenger}.

\begin{defn}[Bounded-turning arcs and Ahlfors regularity]\label{d:bt-ar}
A Jordan arc $\Gamma$ in a metric space is \emph{$C$-bounded turning} if, for every pair of points $x,y\in\Gamma$, the subarc $\Gamma[x,y]$ has diameter at most $C d(x,y)$.  A metric measure space $(Z,d,\mu)$ is \emph{Ahlfors $s$-regular} if there is $C\ge1$ such that
\[
        C^{-1}r^s\le \mu(B(z,r))\le Cr^s
\]
for all $z\in Z$ and all $0<r\le\diam Z$ when $Z$ is bounded, and for all $r>0$ when $Z$ is unbounded.  
\end{defn}

In this paper the measure is always comparable to $s$-dimensional Hausdorff measure on the set under discussion.  For background on Ahlfors regularity, bounded turning, doubling, and quasisymmetric metric geometry, see \cite{Heinonen}.

\subsection{Packing tools}\label{s:prelim-packing}

\begin{defn}[Packing function]\label{d:packing}
For $\varepsilon>0$, let
\[
        \Pack_\varepsilon(B_X(x,R))
\]
denote the \emph{packing number at scale $\varepsilon$} of $B_X(x,R)$, namely the supremum cardinality of an $\varepsilon$-separated subset; it may be infinite.  For a specified basepoint $o\in X$, write
\[
        V_{X,o,\varepsilon}(R)=\Pack_\varepsilon(B_X(o,R))
\]
for the corresponding \emph{basepoint packing function}.  In spaces without bounded geometry this basepoint function is weaker than the \emph{uniform packing function} obtained by taking a supremum over all centers; the distinction is important in \cref{s:qi-diagnostic}.

A geodesic metric space has \emph{bounded geometry} here if for every $R>0$ and every $\varepsilon>0$ one has
\[
        \sup_{x\in X}\Pack_\varepsilon(B_X(x,R))<\infty.
\]
For $Q>0$, we say that $X$ has \emph{polynomial packing exponent at most $Q$} if for every $\varepsilon>0$ there is $C_\varepsilon$ such that
\begin{equation}\label{e:poly-pack}
        \Pack_\varepsilon(B_X(x,R))\le C_\varepsilon(1+R)^Q
\end{equation}
for all $x\in X$ and $R\ge 0$.  We say that $X$ has an \emph{exponential packing bound} if for every $\varepsilon>0$ there are constants $C_\varepsilon\ge1$ and $h_\varepsilon>0$ such that
\begin{equation}\label{e:exp-pack}
        \Pack_\varepsilon(B_X(x,R))\le C_\varepsilon e^{h_\varepsilon R}
\end{equation}
for all $x\in X$ and $R\ge 0$.  For background on packing and covering numbers and on bounded geometry, see \cite{BBI,Roe,CornulierdeLaHarpe}.
\end{defn}

\begin{lem}[Bounded geometry gives exponential packing]\label{l:bg-exp-pack}
Every geodesic metric space with bounded geometry, in the sense of \cref{d:packing}, satisfies an exponential packing bound.
\end{lem}

\begin{proof}
Fix $\varepsilon>0$ and choose $0<r<\varepsilon/4$.  Let $\mathcal N$ be a maximal $r$-separated net in $X$; then $\mathcal N$ is $r$-dense.  By bounded geometry,
\[
        M=\sup_{x\in X}\Pack_r(B_X(x,3r))<\infty.
\]
Let $S$ be an $\varepsilon$-separated subset of $B_X(x,R)$.  For each $s\in S$, choose a geodesic from $x$ to $s$, record the points at times $jr$ with $j\ge0$ and $jr<d_X(x,s)$, and then record the endpoint $s$ as the final entry.  Replace each recorded point by a point of $\mathcal N$ within distance $r$, and include the chain length as part of the code.  There are at most $M$ choices for the first rounded point, since it lies in $B_X(x,r)$.  Consecutive rounded net-points are at distance at most $3r$, so after any rounded point there are at most $M$ possibilities for the next rounded point.  Put $J=\lceil R/r\rceil+2$; every rounded chain has length at most $J$.

If two points of $S$ produced the same code, their final rounded entries would agree, and the original endpoints would be at distance at most $2r<\varepsilon$, contradicting the $\varepsilon$-separation.  Thus the rounded-chain codes are distinct.  Since the chain lengths may vary, we obtain
\[
        |S|\le 1+M+\cdots+M^J
        \le (J+1)\max\{1,M^J\}
        \le C_\varepsilon e^{h_\varepsilon R}.
\]
This is the required exponential packing bound.
\end{proof}

\begin{lem}[Polynomial packing obstruction]\label{l:poly-pack}
Let $I\in\{\N,\Z\}$, let $f:\N\to(0,\infty)$ be nondecreasing, and suppose that $X$ satisfies \eqref{e:poly-pack} with exponent $Q>0$.  If $X$ contains an $f$-distorted sequence indexed by $I$, then
\[
        N^{1/Q}\lesssim_\infty f(N).
\]
In particular, if $f(N)=N^\alpha$, then $\alpha\ge 1/Q$.
\end{lem}

\begin{proof}
Let $(p_n)$ be such a sequence, with distortion data $A,B,N_*$.  For an integer $q\ge1$, consider
\[
        p_1,p_{1+N_*},\ldots,p_{1+qN_*}.
\]
These $q+1$ points lie in $B_X(p_1,Bf(qN_*))$ and are $Af(N_*)$-separated.  Hence
\[
        q+1\le C_{Af(N_*)}\bigl(1+Bf(qN_*)\bigr)^Q.
\]
It follows that $f(qN_*)\ge c q^{1/Q}$ for all sufficiently large $q$ and some $c>0$.  For arbitrary sufficiently large $N$, put $q=\lfloor N/N_*\rfloor$.  Monotonicity gives
\[
        f(N)\ge f(qN_*)\ge c q^{1/Q}\ge c' N^{1/Q}.
\]
The final assertion follows by taking $f(N)=N^\alpha$.
\end{proof}

\begin{lem}[Exponential packing obstruction]\label{l:exp-pack}
Let $I\in\{\N,\Z\}$, let $X$ satisfy \eqref{e:exp-pack}, and let $f:\N\to(0,\infty)$ be nondecreasing.  If $X$ contains an $f$-distorted sequence indexed by $I$, then there are constants $c,C>0$ and $N_0\ge2$ such that
\[
        f(N)\ge c\log N-C
        \qquad(N\ge N_0).
\]
In particular, no $f$-distorted sequence indexed by $I$ exists when $f(N)=o(\log N)$.
\end{lem}

\begin{proof}
Let $A,B,N_*$ be distortion data and put $\varepsilon=Af(N_*)>0$.  For $N\ge N_*$, set $q=\lfloor N/N_*\rfloor$ and consider
\[
        p_1,p_{1+N_*},\ldots,p_{1+qN_*}.
\]
These $q+1$ points are $\varepsilon$-separated, by monotonicity of $f$, and lie in the ball $B_X(p_1,Bf(N))$.  Therefore
\[
        q+1\le \Pack_\varepsilon(B_X(p_1,Bf(N)))
        \le C_\varepsilon e^{h_\varepsilon Bf(N)}.
\]
Since $q+1\asymp N$, taking logarithms and adjusting constants gives
\[
        f(N)\ge c\log N-C
\]
for all sufficiently large $N$.
\end{proof}

\subsection{Snowflaked intervals in asymptotic cones}\label{s:prelim-compactness}

For general background on ultralimits and asymptotic cones, see \cite{GromovAsymptotic,DrutuKapovich,CornulierdeLaHarpe}.

\begin{lem}[The compactness principle]\label{l:compactness}
Let $(X,d)$ be a metric space, let $0<\alpha\le1$, and let $(p_n)_{n\in\N}$ be an $\alpha$-power-distorted sequence in $X$.  Put $r_N=N^\alpha$.  Then:

\begin{enumerate} 
\item For every non-principal ultrafilter $\omega$, the asymptotic cone
\[
        \operatorname{Cone}_\omega(X,p_1,(r_N))
\]
contains a bi-Lipschitz copy of $([0,1],|s-t|^\alpha)$.  
\item If, in addition, there is a fixed proper homogeneous metric space $M$ such that $(X,r_N^{-1}d)$ is isometric to $M$ for every $N$, then $M$ contains a bi-Lipschitz copy of $([0,1],|s-t|^\alpha)$.
\end{enumerate}
\end{lem}

\begin{proof}
Let $A,B,N_0$ be distortion data, and let $B_0$ be the global upper constant from \eqref{e:fdist-upper-all}.  With constant basepoint $p_1$, form the ultralimit of $(X,N^{-\alpha}d,p_1)$. For $t\in[0,1]$, define
\[
        F(t)=\big[(p_{1+\lfloor Nt\rfloor})_{N\ge1}\big]_\omega.
\]
The global upper estimate keeps the representing sequence at uniformly bounded rescaled distance from $p_1$, so $F$ is well defined. If $s\ne t$, then $|\lfloor Ns\rfloor-\lfloor Nt\rfloor|\ge N_0$ for all sufficiently large $N$, and for those $N$,
\[
 A\left(\frac{|\lfloor Ns\rfloor-\lfloor Nt\rfloor|}{N}\right)^\alpha
 \le N^{-\alpha}d(p_{1+\lfloor Ns\rfloor},p_{1+\lfloor Nt\rfloor})
 \le B\left(\frac{|\lfloor Ns\rfloor-\lfloor Nt\rfloor|}{N}\right)^\alpha.
\]
Passing to the ultralimit, and treating $s=t$ trivially, gives
\[
        A|s-t|^\alpha\le d(F(s),F(t))\le B|s-t|^\alpha.
\]
Thus $F$ is the required embedding. For the final assertion, choose an isometry $j_N:(X,r_N^{-1}d)\to M$.  If $j_N(p_1)=q_N$, homogeneity of $M$ provides an isometry $h_N$ of $M$ with $h_N(q_N)=o$.  Hence $h_N\circ j_N$ identifies each pointed space $(X,r_N^{-1}d,p_1)$ with the fixed pointed model $(M,o)$. The cone is then the ultralimit of the constant pointed space $(M,o)$. By properness, every uniformly bounded representing sequence has an $\omega$-limit in a compact ball of $M$, so this ultralimit is canonically isometric to $M$.
\end{proof}

\subsection{Standard background results}\label{s:prelim-background}

The arguments use the following standard inputs.
\begin{enumerate}[label=(\roman*)]
\item Euclidean self-similar chord-arc curves of every Hausdorff dimension $s\in[1,k)$ are constructed, in the needed higher-dimensional form, in \cite[Section~5.2]{GhamsariHerron}; the parametrization used below follows from the characterization in \cite[Theorem~B]{GhamsariHerron}. Every subset of $\R^k$ is doubling, with doubling data depending only on $k$, and the Tukia--V\"ais\"al\"a characterization \cite[Theorem~4.9]{TukiaVaisala} identifies doubling bounded-turning arcs with quasisymmetric images of an interval. V\"ais\"al\"a's Euclidean dimension theorem \cite{Vaisala} then gives the strict dimension drop used at the endpoint. In the plane, Smirnov's sharp quasicircle estimate gives a related stronger result \cite{Smirnov}.
\item Asymptotic cones of geodesic $\delta$-hyperbolic spaces are $\R$-trees; see \cite{BH}. For the complete simply connected manifolds of curvature at most a negative constant that occur here, the relevant cones are universal homogeneous $\R$-trees; see \cite[Theorems~1.1.3 and~1.3.2]{DyubinaPolterovich}.
\item In the diagnostic examples we use that planar geodesic spaces have asymptotic dimension at most $2$, by J{\o}rgensen--Lang \cite[Theorem~2]{JorgensenLang}.
\end{enumerate}

\section{Euclidean spaces}\label{s:euclidean}

Ambient dimension governs the Euclidean case: packing excludes exponents below the critical value, bounded-turning dimension drop excludes the endpoint, and inversions of compact self-similar snowflake curves realize the larger exponents.

We use two standard facts about Euclidean quasiarcs. The existence statement comes from inverting a compact self-similar chord-arc Jordan curve at one point.

\begin{lem}[Euclidean quasiarc facts]\label{l:euclidean-quasiarc}
Let $k\ge 2$.
\begin{enumerate}[label=(\alph*)]
\item For every $1\le s<k$ there are a map $\gamma:\R\to\R^k$ and constants $a,b>0$ such that
\begin{equation}\label{e:euclidean-line}
        a|u-v|^{1/s}\le |\gamma(u)-\gamma(v)|\le b|u-v|^{1/s}
\end{equation}
for all $u,v\in\R$.
\item A bounded-turning arc in $\R^k$ has Hausdorff dimension strictly less than $k$.
\end{enumerate}
\end{lem}

\begin{proof}
For part (a), use a round circle when $s=1$.  When $1<s<k$, the construction in \cite[Section~5.2]{GhamsariHerron} gives a self-similar $s$-dimensional chord-arc Jordan curve $\Gamma\subset\R^k$; see also the self-similar framework of \cite{Hutchinson}.  The chord-arc parametrization characterization \cite[Theorem~B]{GhamsariHerron} gives, in either case, a homeomorphism $\eta:S^1\to\Gamma$ and constants $a_0,b_0>0$ such that
\begin{equation}\label{e:euclidean-circle-snowflake}
        a_0|z-w|^{1/s}\le |\eta(z)-\eta(w)|
        \le b_0|z-w|^{1/s}
        \qquad(z,w\in S^1),
\end{equation}
where $S^1$ has its Euclidean chordal metric.

Fix $z_0\in S^1$ and translate $\Gamma$ so that $\eta(z_0)=0$.  After rotating the circle, take $z_0=(0,1)$ and parametrize $S^1\setminus\{z_0\}$ by
\[
        \sigma(t)=\left(\frac{2t}{1+t^2},\frac{t^2-1}{1+t^2}\right),
        \qquad t\in\R.
\]
Then
\[
 |\sigma(t)-z_0|=\frac{2}{\sqrt{1+t^2}},
 \qquad
 |\sigma(t)-\sigma(u)|=
 \frac{2|t-u|}{\sqrt{1+t^2}\sqrt{1+u^2}}.
\]
Let $I(x)=x/|x|^2$ be inversion in the unit sphere and set
\[
        \gamma(t)=I(\eta(\sigma(t))).
\]
The inversion identity
\[
        |I(x)-I(y)|=\frac{|x-y|}{|x|\,|y|}
\]
together with \eqref{e:euclidean-circle-snowflake} and the two displayed formulas gives
\[
        |\gamma(t)-\gamma(u)|\asymp
        \left(
        \frac{|\sigma(t)-\sigma(u)|}
        {|\sigma(t)-z_0|\,|\sigma(u)-z_0|}
        \right)^{1/s}
        =2^{-1/s}|t-u|^{1/s}.
\]
This proves \eqref{e:euclidean-line}.

For part (b), every subset of $\R^k$ is doubling, so a bounded-turning Euclidean arc is a quasisymmetric image of an interval by the Tukia--V\"ais\"al\"a characterization \cite[Theorem~4.9]{TukiaVaisala}.  Restricting the parametrization to the open interval, V\"ais\"al\"a's Euclidean dimension theorem for quasisymmetric embeddings of Euclidean domains implies that the image, and hence the original arc after adding its two endpoints, has Hausdorff dimension strictly less than $k$ \cite{Vaisala}. In the planar case, Smirnov's quasicircle theorem gives a sharper related estimate \cite{Smirnov}.
\end{proof}

\begin{thm}[Power distortion in $\R^k$]\label{t:euclidean}
Let $k\ge 1$, let $I\in\{\N,\Z\}$, and let $0<\alpha\le 1$.
\begin{enumerate}[label=(\roman*)]
\item If $k=1$, then $\R$ contains an $\alpha$-power-distorted sequence indexed by $I$ if and only if $\alpha=1$.
\item If $k\ge 2$, then $\R^k$ contains an $\alpha$-power-distorted sequence indexed by $I$ if and only if
\[
        \alpha>\frac{1}{k}.
\]
\end{enumerate}
\end{thm}

\begin{proof}
For every $\varepsilon>0$, Euclidean volume comparison gives
\[
        \Pack_\varepsilon(B_{\R^k}(x,R))
        \le C_{\varepsilon,k}(1+R)^k
        \qquad(x\in\R^k,\ R\ge0).
\]
Thus $\R^k$ has polynomial packing exponent at most $k$.  For $k=1$, \cref{l:poly-pack} gives $\alpha\ge 1$; since $\alpha\le 1$, we have $\alpha=1$, realized by $p_n=n$.

Now assume $k\ge2$.  If $\alpha<1/k$, \cref{l:poly-pack} again gives impossibility.

Suppose $\alpha=1/k$.  If an $I$-indexed sequence existed, restrict it to the positive indices when $I=\Z$.  Every rescaling of $\R^k$ is isometric to the fixed proper homogeneous space $\R^k$, so the second assertion of \cref{l:compactness} would give a bi-Lipschitz embedding
\[
        F:([0,1],|s-t|^{1/k})\longrightarrow \R^k.
\]
Its image $\Gamma=F([0,1])$ is an arc.  Moreover it has bounded turning: if $s<t$, then
\[
        \diam F([s,t])\le B(t-s)^{1/k}\le \frac{B}{A}|F(s)-F(t)|.
\]
Snowflaking the interval metric by the exponent $1/k$ multiplies its Hausdorff dimension by $k$, so
\[
        \dim_H([0,1],|s-t|^{1/k})=k.
\]
Bi-Lipschitz maps preserve Hausdorff dimension, and therefore $\dim_H\Gamma=k$.  This dimension equality contradicts \cref{l:euclidean-quasiarc}(b).  Hence $\alpha=1/k$ is impossible.

Finally, let $\alpha>1/k$.  Set $s=1/\alpha$.  Then $1\le s<k$.  By \cref{l:euclidean-quasiarc}(a), there is a map $\gamma:\R\to\R^k$ satisfying
\[
        |\gamma(u)-\gamma(v)|\asymp |u-v|^\alpha.
\]
Thus $p_n=\gamma(n)$ gives a bi-infinite sequence, and its restriction to $n\ge1$ gives a semi-infinite sequence.  This completes the proof.
\end{proof}

\section{Hyperbolic spaces, horocycles, and logarithmic rates}\label{s:hyperbolic}

Geodesic hyperbolicity strongly restricts power rates: in an asymptotic cone, the problem becomes one of embedding snowflaked intervals in an $\R$-tree. Packing instead controls logarithmic rates.

\begin{thm}[Power distortion in geodesic hyperbolic spaces]\label{t:delta-power}
Let $X$ be a geodesic $\delta$-hyperbolic metric space, let $I\in\{\N,\Z\}$, and let $0<\alpha\le1$. If $X$ contains an $\alpha$-power-distorted sequence indexed by $I$, then $\alpha=1$.
\end{thm}

\begin{proof}
Assume $0<\alpha<1$.  When $I=\Z$, restrict the given sequence to its positive-indexed subsequence.  By \cref{l:compactness}, an asymptotic cone of $X$ contains a bi-Lipschitz copy of $([0,1],|s-t|^\alpha)$.  Since $X$ is geodesic and $\delta$-hyperbolic, every asymptotic cone of $X$ is an $\R$-tree.

Let $F:([0,1],|s-t|^\alpha)\to T$ be such an embedding.  Because $F$ is a topological embedding and an $\R$-tree has a unique arc between any two points, $F([0,1])$ is the geodesic segment joining its endpoints.  Hence it is isometric to an interval $[0,L]$, and there is a map $h:[0,1]\to[0,L]$ satisfying
\[
        A|s-t|^\alpha\le |h(s)-h(t)|\le B|s-t|^\alpha.
\]
Since $h$ is continuous and injective, it is monotone.  For the partition $0,1/N,\ldots,1$,
\[
        L=\sum_{j=1}^N |h(j/N)-h((j-1)/N)|
        \ge A N\cdot N^{-\alpha}=A N^{1-\alpha},
\]
which tends to infinity because $\alpha<1$.  The contradiction proves the theorem.
\end{proof}

\begin{prop}[Linear distortion and the sequential boundary]\label{p:hyperbolic-boundary-linear}
Let $X$ be a geodesic $\delta$-hyperbolic metric space, and let $\partial X$ denote its sequential Gromov boundary.  Call a sequence indexed by $\N$ or $\Z$ a \emph{quasi-geodesic sequence} if it is $1$-power-distorted.
\begin{enumerate}
\item The space $X$ contains a semi-infinite quasi-geodesic sequence if and only if $\partial X\ne\varnothing$.
\item The space $X$ contains a bi-infinite quasi-geodesic sequence if and only if $|\partial X|\ge 2$.
\end{enumerate}
\end{prop}

\begin{proof}
By \cite[Remark~2.16]{KapovichBenakli}, the sequential boundary agrees with the quasi-geodesic boundary: every boundary point is represented by a quasi-geodesic ray, and every two distinct boundary points are joined by a bi-infinite quasi-geodesic.  Sampling such a quasi-geodesic at integer parameters gives a $1$-power-distorted sequence, since the additive constants in the quasi-geodesic estimates can be absorbed once the index gap is sufficiently large.  This proves existence in both parts.

Conversely, suppose that $(p_n)$ is a quasi-geodesic sequence.  By \cref{d:distorted} and the global upper estimate \eqref{e:fdist-upper-all}, there are $a,L>0$ and $C\ge0$ such that
\[
        a|n-m|-C\le d(p_n,p_m)\le L|n-m|
\]
for all indices. Join each $p_n$ to $p_{n+1}$ by a geodesic segment and parametrize the resulting concatenated path by arclength.  Let
\[
        x\in[p_m,p_{m+1}],
        \qquad
        y\in[p_n,p_{n+1}],
        \qquad m\le n,
\]
and let $S$ be their distance along the concatenated path.  Since each segment has length at most $L$,
\[
        d_X(x,y)\ge a(n-m)-C-2L,
        \qquad
        S\le L(n-m+2).
\]
Consequently,
\[
        d_X(x,y)\ge \frac{a}{L}S-2a-C-2L.
\]
The lower estimate with one endpoint fixed also gives $d_X(p_m,p_n)\to\infty$ along each infinite tail.  Hence the accumulated arclength of the concatenation is unbounded in every relevant direction, so its arclength parametrization has domain a half-line when $I=\N$ and all of $\R$ when $I=\Z$.  Together with the trivial upper bound $d_X(x,y)\le S$, the displayed estimate shows that the concatenated path is a quasi-geodesic ray or line.  A bi-infinite quasi-geodesic has two distinct ideal endpoints: otherwise its two tails would be asymptotic, contradicting the quasi-geodesic lower bound for points with indices tending to opposite infinities.  Its endpoint, or its two endpoints in the bi-infinite case, belongs to the sequential boundary, again by \cite[Remark~2.16]{KapovichBenakli}.  This proves the converses.
\end{proof}

\begin{cor}[Power distortion in $\Hyp$]\label{c:hyp-power}
Let $I\in\{\N,\Z\}$ and let $0<\alpha\le1$. Then $\Hyp$ contains an $\alpha$-power-distorted sequence indexed by $I$ if and only if $\alpha=1$.
\end{cor}

\begin{proof}
The negative part is \cref{t:delta-power}.  For $\N$, take $p_n=\gamma(n)$ on a unit-speed geodesic ray; for $\Z$, take a complete geodesic line.
\end{proof}

\subsection{Logarithmic distortion in the hyperbolic plane}\label{s:horocycles}

Use the upper half-plane model
\[
        \Hyp=\{(x,y):y>0\},\qquad ds^2=\frac{dx^2+dy^2}{y^2}.
\]

\begin{prop}[Horocycle example]\label{p:horocycle}
For each $I\in\{\N,\Z\}$, the sequence $p_n=(n,1)$, $n\in I$, satisfies
\[
        d_{\Hyp}(p_n,p_m)\asymp \ell(|n-m|)
        \qquad(m,n\in I,\ m\ne n).
\]
\end{prop}

\begin{proof}
For points $(x,1),(x',1)$ in the upper half-plane,
\[
        \cosh d_{\Hyp}((x,1),(x',1))=1+\frac{(x-x')^2}{2}.
\]
Equivalently,
\[
        d_{\Hyp}((x,1),(x',1))=2\arsinh \frac{|x-x'|}{2}.
\]
Since
\[
        \arsinh u=\log(u+\sqrt{1+u^2}),
\]
this is uniformly comparable to $\log(1+|x-x'|)$ for $|x-x'|\ge 0$.
\end{proof}

\begin{cor}[No sublogarithmic distortion under exponential packing]\label{c:no-sublog-exp}
Let $X$ be a metric space satisfying the exponential packing bound \eqref{e:exp-pack}, let $I\in\{\N,\Z\}$, and let $f:\N\to(0,\infty)$ be nondecreasing with
\[
        f(N)=o(\ell(N)).
\]
Then $X$ contains no $f$-distorted sequence indexed by $I$.
\end{cor}

\begin{proof}
By \cref{l:exp-pack}, every such distorted sequence would satisfy $f(N)\ge c\log N-C$ for all sufficiently large $N$.  Since $\ell(N)=\log(1+N)\asymp\log N$, this contradicts $f(N)=o(\ell(N))$.
\end{proof}

\begin{cor}[Finitely generated groups and word-hyperbolic groups]\label{c:fg-no-sublog}
Let $G$ be a finitely generated group with a word metric, let $I\in\{\N,\Z\}$, and let $f:\N\to(0,\infty)$ be nondecreasing with
\[
        f(N)=o(\ell(N)).
\]
Then $G$ contains no $f$-distorted sequence indexed by $I$.
\end{cor}

In particular, \cref{c:fg-no-sublog} applies to every finitely generated word-hyperbolic group.

\begin{proof}
A Cayley graph of a finitely generated group has bounded valence, so its vertex balls grow at most exponentially and its word metric satisfies \eqref{e:exp-pack}.  Apply \cref{l:exp-pack}; hyperbolicity is not needed.
\end{proof}

\begin{cor}[Sharpness in $\Hyp$]\label{c:log-sharp-hyp}
Let $I\in\{\N,\Z\}$.
\begin{enumerate}[label=\textup{(\roman*)}]
\item The space $\Hyp$ contains an $\ell$-distorted sequence indexed by $I$.
\item If $f:\N\to(0,\infty)$ is nondecreasing and satisfies $f(N)=o(\ell(N))$, then $\Hyp$ contains no $f$-distorted sequence indexed by $I$.
\end{enumerate}
\end{cor}

\begin{proof}
Hyperbolic balls have area
\[
        \Vol(B_{\Hyp}(R))=2\pi(\cosh R-1)\asymp e^R.
\]
For fixed $\varepsilon>0$, the $\varepsilon/2$-balls centered at an $\varepsilon$-separated subset of $B_{\Hyp}(R)$ are disjoint and lie in $B_{\Hyp}(R+\varepsilon/2)$. The area formula therefore gives an exponential packing bound for $\Hyp$.  Apply \cref{c:no-sublog-exp} for the negative statement and \cref{p:horocycle} for the positive statement.
\end{proof}

\subsection{Arbitrary \texorpdfstring{$\delta$}{delta}-hyperbolic spaces and packing}\label{s:arbitrary-hyp}

\begin{cor}\label{c:hyperbolic-exp-packing}
Let $X$ be a geodesic $\delta$-hyperbolic space satisfying an exponential packing bound, let $I\in\{\N,\Z\}$, and let $f:\N\to(0,\infty)$ be nondecreasing with $f(N)=o(\ell(N))$. Then $X$ contains no $f$-distorted sequence indexed by $I$.
\end{cor}

Examples covered by \cref{c:hyperbolic-exp-packing} include bounded-degree hyperbolic graphs, Cayley graphs of finitely generated word-hyperbolic groups, and universal covers of compact negatively curved manifolds.

\begin{proof}
Only the packing estimate is used; apply \cref{l:exp-pack}.
\end{proof}

The packing hypothesis cannot be removed.

\begin{ex}[A CAT$(-1)$ surface with log--log horocycles]\label{ex:loglog}
Let $X=\R^2$ with coordinates $(x,t)$ and metric
\[
        ds^2=dt^2+e^{-2\Phi(t)}dx^2,
        \qquad \Phi(t)=t+e^t.
\]
Then:
\begin{enumerate}[label=\textup{(\roman*)}]
\item $X$ is complete and simply connected, and its Gaussian curvature satisfies
\[
        K(t)=\Phi''(t)-(\Phi'(t))^2=e^t-(1+e^t)^2\le -1.
\]
Thus $X$ is CAT$(-1)$, hence $\delta$-hyperbolic.
\item For each $I\in\{\N,\Z\}$, the sequence $p_n=(n,0)$, $n\in I$, satisfies
\[
        d_X(p_n,p_m)\asymp \log\log(e^e+|n-m|)
        \qquad(m,n\in I,\ m\ne n).
\]
\end{enumerate}
\end{ex}

\begin{proof}
A finite-length curve has bounded $t$-coordinate because its total $t$-variation is at most its length.  If $t$ stays in a compact interval $[-T,T]$, put
\[
        c_T=\min_{[-T,T]}e^{-\Phi(t)}>0.
\]
For every absolutely continuous curve $\sigma(s)=(x(s),t(s))$ in this strip,
\[
        \operatorname{length}(\sigma)
        \ge \int e^{-\Phi(t(s))}|x'(s)|\,ds
        \ge c_T\int |x'(s)|\,ds
        \ge c_T|\Delta x|.
\]
Thus $|x|\to\infty$ also forces infinite length.  The divergent-path criterion gives completeness, and Hopf--Rinow gives properness.  For a warped product metric $dt^2+a(t)^2dx^2$ with $a(t)=e^{-\Phi(t)}$, the curvature is
\[
        K=-\frac{a''(t)}{a(t)}=\Phi''(t)-(\Phi'(t))^2.
\]
Thus $K\le -1$, and the Cartan--Hadamard theorem gives CAT$(-1)$.  The same estimate verifies the curvature hypothesis used later when applying the asymptotic-cone theorem of Dyubina--Polterovich \cite[Theorem~1.3.2]{DyubinaPolterovich} to this surface.

Since the metric coefficients are independent of $x$, horizontal translations in $x$ are isometries, so it suffices to estimate
\[
        D(N)=d_X((0,0),(N,0)).
\]
For sufficiently large $N$, choose the unique $T\ge 0$ satisfying $\Phi(T)=\log N$.  The path going vertically from $t=0$ to $t=T$, horizontally from $x=0$ to $x=N$ at height $T$, and vertically back to $t=0$ has length
\[
        2T+N e^{-\Phi(T)}=2T+1.
\]
Since $T+e^T=\log N$, one has $T=\log\log N+O(1)$: indeed, $e^T\le\log N$ gives $T\le\log\log N$, while $T\le e^T$ for $T\ge0$ gives $\log N\le2e^T$ and hence $T\ge\log\log N-\log 2$.  Thus $D(N)\lesssim \log\log(e^e+N)$, after adjusting constants for bounded $N$.

For the lower bound, let $\sigma$ be any path from $(0,0)$ to $(N,0)$ of length $L$, and let $T$ be the maximum $t$-coordinate reached by $\sigma$.  Since both endpoints have $t$-coordinate $0$, reaching height $T$ and returning requires total variation at least $2T$ in the $t$-coordinate.  Hence $T\le L/2$.  Since $\Phi'(t)=1+e^t>0$ for all $t\in\R$, and every point of $\sigma$ has $t$-coordinate at most $T$, the horizontal displacement satisfies
\[
        N\le e^{\Phi(T)}L\le e^{\Phi(L/2)}L.
\]
Taking logarithms,
\[
        \log N\le \log L+\Phi(L/2)
        =\log L+\frac{L}{2}+e^{L/2}.
\]
For $L\ge 2$, the right-hand side is at most $C e^{L/2}$ for an absolute constant $C$.  Therefore $\log\log N\le L/2+\log C$, and hence
\[
        L\ge 2\log\log N-C'
\]
whenever $N$ is sufficiently large.  Adjusting constants for bounded $N$ gives $D(N)\gtrsim \log\log(e^e+N)$.
\end{proof}

\begin{rem}
Thus $\delta$-hyperbolicity rules out sublinear power rates because asymptotic cones are trees, but it does not itself give a logarithmic lower bound for general rates; that bound comes from packing.
\end{rem}

\section{Simplicial trees}\label{s:simplicial-trees}

Throughout this section, every simplicial tree has the path metric with unit-length edges.  No local finiteness or degree bound is assumed unless explicitly stated.

\begin{thm}[Unbounded sequence distortion in simplicial trees is linear]\label{t:simplicial-trees}
Let $T$ be a simplicial tree, let $I\in\{\N,\Z\}$, and let $f:\N\to(0,\infty)$ be nondecreasing and unbounded. Suppose that $(p_n)_{n\in I}$ is $f$-distorted in $T$. Then:
\begin{enumerate}[label=\textup{(\roman*)}]
\item the map $n\mapsto p_n$ is a quasi-isometric embedding of $I$ into $T$; more precisely, there are constants $c,L>0$ and $C\ge 0$ such that
\[
        c|n-m|-C\le d_T(p_n,p_m)\le L|n-m|
        \qquad(m,n\in I);
\]
\item one has
\[
        f(N)\asymp N.
\]
\end{enumerate}
\end{thm}

\begin{proof}
First consider $I=\N$. Let $A,B,N_0$ be distortion data, and let $L>0$ be the bounded-increment constant supplied by \cref{l:bounded-increments}.  Thus consecutive points satisfy $d_T(p_n,p_{n+1})\le L$, and the triangle inequality already gives
\[
        d_T(p_i,p_k)\le L(k-i)\qquad(i<k).
\]
For the linear lower bound, first prove bounded backtracking.  Since $f$ is nondecreasing and unbounded, choose an integer $M_0\ge N_0$ such that
\[
        A f(M_0)>2L.
\]
We claim that for all $i<j<k$,
\begin{equation}\label{e:tree-bounded-backtracking}
        d_T(p_j,[p_i,p_k])\le D:=L(M_0+2),
\end{equation}
where $[p_i,p_k]$ denotes the geodesic segment from $p_i$ to $p_k$.

Indeed, let $b$ be the projection of $p_j$ to $[p_i,p_k]$ and set $h=d_T(p_j,b)$.  If $h=0$ there is nothing to prove.  Otherwise the geodesics $[p_i,p_j]$ and $[p_j,p_k]$ both contain the segment $[b,p_j]$.  Let $y\in[b,p_j]$ satisfy $d_T(y,p_j)=h/2$.  Join each consecutive pair $p_r,p_{r+1}$ by the geodesic segment between them.  Since deleting $y$ separates $p_j$ from both $p_i$ and $p_k$, the resulting path from $p_i$ to $p_j$ must cross $y$, and the path from $p_j$ to $p_k$ must also cross $y$.  Hence there are indices $r\in\{i,\ldots,j-1\}$ and $s\in\{j,\ldots,k-1\}$ such that $y\in[p_r,p_{r+1}]$ and $y\in[p_s,p_{s+1}]$.  Therefore
\[
        d_T(p_r,p_s)\le 2L.
\]
By the lower distortion estimate and the choice of $M_0$, this implies $s-r<M_0$.  The part of the discrete path from $p_r$ to $p_j$ has length at most $L(j-r)$ but must move from within distance $L$ of $y$ to $p_j$, so
\[
        L(j-r)\ge h/2-L.
\]
Similarly $L(s-j)\ge h/2-L$.  Hence
\[
        L(s-r)\ge h-2L.
\]
Since $s-r<M_0$, we get $h\le L(M_0+2)=D$.  This proves \eqref{e:tree-bounded-backtracking}.

Now fix $i<k$.  By \eqref{e:tree-bounded-backtracking}, every intermediate point $p_j$, $i\le j\le k$, lies in the $D$-neighborhood of the geodesic segment $[p_i,p_k]$.  Choose an integer $M_1\ge N_0$ such that
\[
        A f(M_1)>2D+2.
\]
Choose at most $d_T(p_i,p_k)+2$ unit-spaced centers covering $[p_i,p_k]$.  Project each $p_j$, $i\le j\le k$, to $[p_i,p_k]$ and assign it to a covering center within distance $1$ of that projection.  By \eqref{e:tree-bounded-backtracking}, an assigned point lies in the ball of radius $D+1$ about its center.  Any one of these balls contains at most $M_1$ of the points $p_i,p_{i+1},\ldots,p_k$.  Indeed, if it contained $M_1+1$ selected points, the first and last corresponding indices would differ by at least $M_1$, while their points would be at distance at most $2D+2$, contradicting the lower distortion estimate.  Therefore
\[
        k-i+1\le M_1(d_T(p_i,p_k)+2).
\]
Equivalently,
\[
        d_T(p_i,p_k)\ge M_1^{-1}(k-i+1)-2.
\]
This is the required quasi-geodesic lower bound.

Finally, compare $f$ with the linear rate.  The adjacent-step bound gives
\[
        d_T(p_1,p_{1+N})\le LN.
\]
For $N\ge N_0$, the lower distortion inequality therefore gives $f(N)\le (L/A)N$.  Conversely, for all sufficiently large $N$, the quasi-geodesic lower bound and the upper distortion inequality give
\[
        B f(N)\ge d_T(p_1,p_{1+N})\ge cN-C,
\]
so $f(N)\ge c'N-C'$ for all sufficiently large $N$.  Thus $f(N)\asymp N$.  The proof for $\Z$ is identical after restricting to a finite interval $i<j<k$ of indices.
\end{proof}

\begin{cor}[Power distortion in simplicial trees]\label{c:tree-powers}
Let $T$ be a simplicial tree with all edges of length $1$.
\begin{enumerate}[label=(\roman*)]
\item If $0<\alpha<1$ and $I\in\{\N,\Z\}$, then $T$ contains no $\alpha$-power-distorted sequence indexed by $I$.
\item The tree $T$ contains a $1$-power-distorted sequence indexed by $\N$ if and only if $T$ contains a geodesic ray.
\item The tree $T$ contains a $1$-power-distorted sequence indexed by $\Z$ if and only if $T$ contains a bi-infinite geodesic line.
\end{enumerate}
\end{cor}

\begin{proof}
Part (i) follows from \cref{t:simplicial-trees}, since $N^\alpha$ is unbounded and sublinear when $0<\alpha<1$.

If $T$ contains a geodesic ray $\gamma:[0,\infty)\to T$, then $p_n=\gamma(n)$ gives a $1$-power-distorted semi-infinite sequence.  Conversely, let $(p_n)$ be a $1$-power-distorted semi-infinite sequence.  By \cref{t:simplicial-trees}, it is a quasi-geodesic ray and satisfies the bounded-backtracking estimate \eqref{e:tree-bounded-backtracking}.  If $1<j<k$, let $q$ be the projection of $p_j$ to $[p_1,p_k]$.  Then $d_T(p_j,q)\le D$, and the tree identity
\[
        [p_1,p_j]\cap[p_1,p_k]=[p_1,q]
\]
shows that the common initial segment has length
\[
        d_T(p_1,q)=d_T(p_1,p_j)-d_T(p_j,q)
        \ge d_T(p_1,p_j)-D.
\]  The quantity tends to infinity with $j$.  Thus, for every $R$, all sufficiently long segments $[p_1,p_n]$ agree on their initial segment of length $R$; the union of these stabilized segments is a geodesic ray.

For a bi-infinite quasi-geodesic, apply the preceding stabilization argument to the segments $[p_0,p_n]$ as $n\to+\infty$ and as $n\to-\infty$.  This gives limiting rays $\rho_+$ and $\rho_-$ based at $p_0$.  More explicitly, for fixed $j>0$ and $n>j$, the bounded-backtracking estimate gives
\[
        d_T(p_j,[p_0,p_n])\le D.
\]
As $n\to+\infty$, the relevant initial segments of $[p_0,p_n]$ stabilize to $\rho_+$, and therefore $d_T(p_j,\rho_+)\le D$.  The same argument for $j<0$ and $n\to-\infty$ gives $d_T(p_j,\rho_-)\le D$.  Thus the positive and negative tails lie in the $D$-neighborhoods of their respective limiting rays.  These rays determine distinct ends.  Indeed, suppose that both tails determine the same end and let $\rho$ represent it.  Choose $H$ so that the two tails lie in the $H$-neighborhood of $\rho$; enlarging $H$ to cover the finitely many remaining points, assume that $d_T(p_n,\rho)\le H$ for every $n$.  Let $L$ bound the distances between consecutive points, project $p_n$ to $x_n\in\rho$, and put $a_n=d_T(\rho(0),x_n)$.  Then $a_n\to\infty$ as $n\to\pm\infty$ and $|a_{n+1}-a_n|\le L+2H$.  For $R>a_0$, choose the first positive index $i_R$ and the largest negative index $j_R$ for which $a_{i_R},a_{j_R}\ge R$.  Then both radial coordinates are less than $R+L+2H$, so $d_T(p_{i_R},p_{j_R})\le L+4H$, whereas $i_R-j_R\to\infty$.  The uniform bound contradicts the quasi-geodesic lower bound.  Thus the two ends are distinct, and the geodesic joining them is a bi-infinite line.
\end{proof}

\begin{cor}[Logarithmic and sublogarithmic distortion in trees]\label{c:tree-log}
Let $T$ be any simplicial tree and let $I\in\{\N,\Z\}$.
\begin{enumerate}[label=\textup{(\roman*)}]
\item The tree $T$ contains no $\ell$-distorted sequence indexed by $I$.
\item More generally, if $f:\N\to(0,\infty)$ is nondecreasing, unbounded, and satisfies $f(N)=o(N)$, then $T$ contains no $f$-distorted sequence indexed by $I$.
\end{enumerate}
\end{cor}

No degree bound on $T$ is assumed in \cref{c:tree-log}.

\begin{proof}
Both assertions follow from \cref{t:simplicial-trees}, since every nondecreasing unbounded distortion function realized in a tree satisfies $f(N)\asymp N$.
\end{proof}

\begin{rem}[Role of degree]\label{r:tree-degree}
The tree obstruction requires no degree bound. For bounded-degree $T$, exponential packing already rules out sublogarithmic distortion, but the tree argument is stronger: it also rules out logarithmic distortion and, in fact, every unbounded sublinear distortion function.  If $T$ has unbounded or infinite degree, exponential packing can fail badly; for instance an infinite star has infinitely many $1$-separated leaves in a ball of radius $2$.  Infinite valence can allow bounded distortion; for example, listing the leaves of an infinite star gives an $f$-distorted sequence for a bounded function $f\asymp 1$.  It still cannot produce an unbounded sublinear rate such as $\log N$ or $\log\log N$, because the bottleneck argument above uses only the tree structure and the uniform bound on consecutive distances supplied by \cref{l:bounded-increments}.
\end{rem}

\section{A quasi-isometry diagnostic example}\label{s:qi-diagnostic}

Sequence distortion can detect coarse information missed by packing, asymptotic dimension, and asymptotic cones. We give two examples.  The first lacks bounded geometry and separates spaces by a log--log-distorted sequence.  Unbounded geometry is essential for this rate: by \cref{c:fg-no-sublog}, no finitely generated group with a word metric can contain a sublogarithmically distorted sequence.  The second has bounded geometry and detects logarithmic distortion.

Put
\[
        f_0(0)=0,
        \qquad
        f_0(N)=\log\log(e^e+N)\quad(N\ge1).
\]
The restriction of $f_0$ to $\N$ is a rate function, so \cref{l:qi-invariance} applies.

\begin{lem}[Finite wedge tail lemma]\label{l:wedge-tail}
Let
\[
        Z=Z_1\vee\cdots\vee Z_k
\]
be a finite wedge of geodesic metric spaces, with common wedge point $o$ and the induced path metric, and let $f$ be nondecreasing and unbounded.
\begin{enumerate}[label=\textup{(\roman*)}]
\item If $(z_n)_{n\in\N}$ is $f$-distorted in $Z$, then there exist $i\in\{1,\ldots,k\}$ and $N_0\ge1$ such that
\[
        z_n\in Z_i\qquad(n\ge N_0).
\]
\item If $(z_n)_{n\in\Z}$ is $f$-distorted in $Z$, then there exist $i_+,i_-\in\{1,\ldots,k\}$ and $N_0\ge1$ such that
\[
        z_n\in Z_{i_+}\quad(n\ge N_0),
        \qquad
        z_n\in Z_{i_-}\quad(n\le -N_0).
\]
\end{enumerate}
\end{lem}

\begin{proof}
Let $A,B,N_*$ be distortion data, and let $L>0$ be the bounded-increment constant supplied by \cref{l:bounded-increments}.  Regard the wedge point $o$ as belonging to every factor, and call $n$ a \emph{factor-change index} when $z_n,z_{n+1}\ne o$ lie in distinct factors.  At every factor change,
\[
        d_Z(z_n,o)+d_Z(o,z_{n+1})=d_Z(z_n,z_{n+1})\le L.
\]
Thus both points lie in the fixed ball $B_Z(o,L)$.  An occurrence of $z_n=o$ also lies in this ball and can be regarded as a transition between the adjacent non-wedge terms.

Every fixed bounded ball contains only finitely many terms of the sequence.  Indeed, infinitely many such terms would contain pairs with arbitrarily large index gap, whereas for $|i-j|\ge N_*$,
\[
        d_Z(z_i,z_j)\ge Af(|i-j|)\longrightarrow\infty.
\]
Thus there are only finitely many factor changes and visits to $o$, so a tail lies in one factor.  Applying the same argument to the positive and negative tails proves the bi-infinite assertion.
\end{proof}

\begin{lem}[Asymptotic cones of finite wedges]\label{l:wedge-cones}
Let $Z=Z_1\vee\cdots\vee Z_k$ be a finite one-point wedge, based at the common wedge point $o$.  For every sequence $(r_i)$ with $r_i\to\infty$ and every non-principal ultrafilter $\omega$,
\[
        \operatorname{Cone}_\omega(Z,o,(r_i))
        \cong
        \operatorname{Cone}_\omega(Z_1,o,(r_i))\vee\cdots\vee
        \operatorname{Cone}_\omega(Z_k,o,(r_i)).
\]
\end{lem}

\begin{proof}
A point of the cone is represented by a sequence $(x_i)$ with $d(x_i,o)=O(r_i)$.  Assign each occurrence of $o$ to one fixed factor.  The sets of indices on which $x_i$ lies in the resulting respective factors form a finite partition, so the ultrafilter selects one of them.  Replacing the remaining terms by the wedge point identifies the represented cone point with a point of the corresponding factor cone.  This is well defined: if two equivalent representing sequences lie in distinct factors for $\omega$-almost every $i$, then the wedge distance formula and $r_i^{-1}d(x_i,y_i)\to_\omega0$ force both $r_i^{-1}d(x_i,o)$ and $r_i^{-1}d(y_i,o)$ to tend to $0$, so both sequences represent the wedge point.  If the cone point is the wedge point, the same representative may of course be assigned to any factor.  If two representing sequences lie in different factors for $\omega$-almost every $i$, then
\[
        d_Z(x_i,y_i)=d_Z(x_i,o)+d_Z(o,y_i),
\]
and this equality passes to the ultralimit.  Hence cross-factor distances are exactly the path-metric distances in the one-point wedge, proving the asserted isometry.
\end{proof}

We now construct the spaces.  Let $W$ denote the warped plane from \cref{ex:loglog}, namely
\[
        W=(\R^2,\,dt^2+e^{-2\Phi(t)}dx^2),
        \qquad \Phi(t)=t+e^t.
\]
It is CAT$(-1)$ and the horocyclic sequence $(n,0)$ is $f_0$-distorted.  Let $T_*$ be the rooted simplicial tree in which every vertex at level $j$ has
\[
        b_j=\left\lceil \exp(\exp(\exp(j+10)))\right\rceil
\]
children.  The tree $T_*$ is locally finite, hence proper, but has unbounded degree and extremely fast packing growth.  We use its root as the distinguished point in the wedges below.  Finally define pointed wedge spaces
\[
        X=T_*\vee \Hyp,
        \qquad
        Y=T_*\vee \Hyp\vee W,
\]
by identifying chosen basepoints and taking the induced path metric.

\begin{prop}[Sequence distortion distinguishes $X$ and $Y$]\label{p:seq-distinguishes}
For each $I\in\{\N,\Z\}$, the space $Y$ contains an $f_0$-distorted sequence indexed by $I$.  The space $X$ contains no $f_0$-distorted sequence indexed by $\N$, and hence none indexed by $\Z$.  Consequently, $X$ and $Y$ are not quasi-isometric.
\end{prop}

\begin{proof}
The factor $W$ gives the positive statement for $Y$: the sequence $(n,0)$, indexed by either $\N$ or $\Z$, is $f_0$-distorted by \cref{ex:loglog}.

For the negative statement, apply \cref{l:wedge-tail} to $X=T_*\vee\Hyp$ and $f=f_0$. A tail of any hypothetical $f_0$-distorted sequence in $X$ lies in $T_*$ or $\Hyp$ and, after reindexing, remains $f_0$-distorted in that factor.  It cannot lie in $T_*$, because \cref{c:tree-log} rules out every unbounded sublinear distortion function in a simplicial tree.  It cannot lie in $\Hyp$, because \cref{c:log-sharp-hyp} rules out sublogarithmic distortion in the hyperbolic plane.  Hence $X$ has no $f_0$-distorted sequence.

If $X$ and $Y$ were quasi-isometric, \cref{l:qi-invariance} would transfer the $f_0$-distorted sequence in $Y$ to one in $X$, a contradiction.
\end{proof}

We next compare standard large-scale features of $X$ and $Y$. Since $T_*$ has unbounded degree, neither space has bounded geometry, so item~(ii) compares basepoint rather than uniform packing growth.

\begin{prop}[Large-scale comparison for the first pair]\label{p:standard-invariants-agree}
Let $o$ denote the wedge point in each of the spaces $X$ and $Y$ defined above.
\begin{enumerate}[label=\textup{(\roman*)}]
\item There exists $\delta\ge0$ such that $X$ and $Y$ are proper geodesic simply connected $\delta$-hyperbolic spaces.
\item For $Z\in\{X,Y,T_*\}$, put
\[
        V_Z(R)=\Pack_1(B_Z(o,R)).
\]
Then
\[
        V_X(R)\asymp V_Y(R)\asymp V_{T_*}(R).
\]
Neither $X$ nor $Y$ has bounded geometry.
\item Their asymptotic dimensions satisfy
\[
        \asdim X=\asdim Y=2.
\]
\item Their Lipschitz filling functions satisfy
\[
        \delta_X\filleq L,
        \qquad
        \delta_Y\filleq L.
\]
\item For every sequence $(r_i)$ with $r_i\to\infty$ and every non-principal ultrafilter $\omega$,
\[
        \operatorname{Cone}_\omega(X,o,(r_i))
        \cong
        \operatorname{Cone}_\omega(Y,o,(r_i)).
\]
\end{enumerate}
\end{prop}

In particular, the cones in item~(v) have the same homeomorphism type and the same bi-Lipschitz type.

\begin{proof}
The wedge of finitely many proper geodesic spaces at a point is again proper and geodesic.  Each of the pointed factors $T_*$, $\Hyp$, and $W$ contracts to its chosen basepoint by geodesic contraction, so the factorwise contractions combine to contract each finite wedge.  The factors are respectively $0$-hyperbolic, CAT$(-1)$, and CAT$(-1)$.  After deleting any common stems through the wedge point, a geodesic triangle in a one-point wedge reduces either to a geodesic triangle in a single factor or to a tripod.  Hence the wedge is $\delta$-hyperbolic with $\delta$ depending only on the hyperbolicity constants of the factors.  This proves (i).

For (ii), a one-point wedge of finitely many factors $Z_i$ satisfies
\[
        \max_i V_{Z_i}(R)\le V_{\bigvee_i Z_i}(R)
        \le \sum_i V_{Z_i}(R).
\]
A separated set in one factor gives the lower bound; intersecting a separated set with the factors gives the upper bound.  The hyperbolic plane has $V_{\Hyp}(R)\le C e^{CR}$.  The warped plane $W$ has a crude bound
\[
        V_W(R)\le \exp(Ce^R).
\]
Indeed, a ball of radius $R$ about $(0,0)$ lies in the coordinate rectangle $|t|\le R$ and $|x|\le R e^{\Phi(R)}$.  Fix $0<\delta<1/3$ and partition this rectangle into boxes of $t$-height at most $\delta$ and $x$-width at most $\delta e^{\Phi(-R)}$.  Since $e^{-\Phi(t)}\le e^{-\Phi(-R)}$ on the rectangle, joining two points of one box vertically and then horizontally gives distance at most $2\delta<1$.  The number of boxes is at most
\[
        C(1+R)^2e^{\Phi(R)-\Phi(-R)}\le \exp(C'e^R).
\]
Thus each box contains at most one point of a $1$-separated set.  On the other hand, put $n=\lfloor R\rfloor$.  For $n\ge1$, the number of vertices of $T_*$ in the ball of radius $R$ is at least
\[
        \prod_{j=0}^{n-1} b_j,
\]
and therefore
\[
        \log V_{T_*}(R)\ge \log b_{n-1}
        \ge \exp(\exp(n+9)).
\]
By contrast, $\log V_W(R)\le C'e^R$ and $\log V_{\Hyp}(R)\le CR+O(1)$.  Thus $V_{\Hyp}(R)+V_W(R)=o(V_{T_*}(R))$, and, after changing constants on bounded intervals,
\[
        V_X(R)\asymp V_{T_*}(R)\asymp V_Y(R).
\]
Thus basepoint packing does not distinguish the spaces; their uniform packing functions are already infinite at bounded radii because $T_*$ has unbounded degree.

For (iii), $\asdim T_*=1$ and $\asdim \Hyp=2$.  The warped plane $W$ is a planar geodesic metric space, hence has asymptotic dimension at most $2$ by J{\o}rgensen--Lang's theorem on planar geodesic spaces \cite[Theorem~2]{JorgensenLang}.  The finite union theorem for asymptotic dimension gives
\[
        \asdim X=\asdim Y=2,
\]
because both spaces contain the factor $\Hyp$.

For (iv), use the Lipschitz filling function from \cref{d:filling}.  The factors $T_*$, $\Hyp$, and $W$ are CAT$(-1)$, and their one-point wedges are CAT$(-1)$ by the Reshetnyak gluing theorem, since a point is a complete convex subset; see \cite[Chapter~II.11]{BH}.  The hyperbolic isoperimetric inequality therefore gives a linear upper bound for the filling functions of $X$ and $Y$: loops of length $L$ admit fillings of area at most $CL+C$ for a uniform constant $C$.

Conversely, collapsing all other factors to the wedge point gives a $1$-Lipschitz retraction onto $\Hyp$.  Composing a Lipschitz filling with this retraction cannot increase its parametrized area.  A geodesic circle of radius $r$ in $\Hyp$ has length $2\pi\sinh r$ and bounds a disk of area $2\pi(\cosh r-1)$, so for $r\ge1$ the disk area is comparable to the boundary length.  Apply \eqref{e:lipschitz-stokes} to a primitive of the hyperbolic area form and then use \eqref{e:area-form-bound}.  Every Lipschitz filling of the positively oriented circle therefore has area at least the enclosed disk area.  Thus the hyperbolic-plane filling function has a linear lower bound, which persists in each wedge. Hence
\[
        \delta_X\filleq L,
        \qquad
        \delta_Y\filleq L.
\]

For (v), fix the wedge point, a scaling sequence $(r_i)$ tending to infinity, and a non-principal ultrafilter $\omega$.  By \cref{l:wedge-cones}, the cone of a one-point wedge at its wedge point is the wedge of the factor cones.  Let $\mathcal T$ denote the asymptotic cone of $T_*$ for this choice of scaling and ultrafilter.  It is an $\R$-tree.  By the theorem of Dyubina--Polterovich \cite[Theorem~1.3.2]{DyubinaPolterovich}, every asymptotic cone of a complete simply connected Riemannian manifold whose sectional curvature is bounded above by a negative constant is the complete homogeneous $\R$-tree $\mathcal U$ of valence $2^{\aleph_0}$.  Since both $\Hyp$ and $W$ satisfy $K\le -1$, their asymptotic cones contribute copies of $\mathcal U$.  Therefore
\[
        \operatorname{Cone}_\omega(X,o,(r_i))\cong \mathcal T\vee\mathcal U,
        \qquad
        \operatorname{Cone}_\omega(Y,o,(r_i))\cong \mathcal T\vee\mathcal U\vee\mathcal U.
\]
The wedge $\mathcal U\vee\mathcal U$ is a complete $\R$-tree, and every point still has valence $2^{\aleph_0}$.  By the uniqueness of the complete homogeneous $\R$-tree of this valence \cite[Theorem~1.1.3(ii)]{DyubinaPolterovich}, it is isometric to $\mathcal U$. By homogeneity, the isometry may be chosen to send the wedge point to the distinguished basepoint of $\mathcal U$.  Consequently
\[
        \operatorname{Cone}_\omega(X,o,(r_i))
        \cong \operatorname{Cone}_\omega(Y,o,(r_i))
\]
for every ultrafilter and every scaling sequence, with the wedge point as basepoint.
\end{proof}

\begin{rem}[Why the first example uses unbounded geometry]\label{r:unbounded-geometry-needed}
The log--log example uses unbounded geometry twice.  The tree $T_*$ has unbounded valence, which makes its basepoint coarse packing growth dominate the additional warped plane, and the warped plane has unbounded negative curvature, which allows a horocycle to realize the rate $\log\log N$.  In bounded-valence graphs, and in particular in finitely generated groups with word metrics, \cref{c:fg-no-sublog} rules out sublogarithmic distortion altogether.  Sublogarithmic distortion is therefore unavailable in bounded geometry, so the next construction uses the sharp logarithmic scale.
\end{rem}

\subsection{A bounded-geometry logarithmic variant}\label{ss:bounded-geometry-diagnostic}

Under bounded geometry, the preceding example cannot retain the sublogarithmic rate $f_0$: by \cref{l:bg-exp-pack}, the local-packing definition gives an exponential packing bound through chains of uniformly bounded local nets along geodesics. The analogue is therefore distinguished by logarithmic distortion.

Let $T_D$ be the $D$-regular simplicial tree with unit length edges, where $D\ge 5$, and let $E=\R^2$ with its Euclidean metric.  Define
\[
        X_b=T_D\vee E,
        \qquad
        Y_b=T_D\vee E\vee\Hyp,
\]
by identifying chosen basepoints and taking the induced path metric.

\begin{prop}[A bounded-geometry pair detected by logarithmic distortion]\label{p:bounded-log-distinguishes}
The spaces $X_b$ and $Y_b$ are proper geodesic simply connected spaces of bounded geometry. For each $I\in\{\N,\Z\}$, the space $Y_b$ contains an $\ell$-distorted sequence indexed by $I$. The space $X_b$ contains no $\ell$-distorted sequence indexed by $\N$, and hence none indexed by $\Z$. Consequently, $X_b$ and $Y_b$ are not quasi-isometric.
\end{prop}

\begin{proof}
A finite wedge of proper geodesic spaces is proper and geodesic.  Each of the pointed factors $T_D$, $E$, and $\Hyp$ contracts to its chosen basepoint by geodesic contraction, so the factorwise contractions combine to contract each finite wedge.  For bounded geometry, fix $r,R>0$ and bound uniformly in the center the sizes of $r$-separated subsets of $R$-balls.  The regular tree has this property because it has bounded valence, and the two manifolds $E$ and $\Hyp$ have it by homogeneity.  A finite wedge of bounded-geometry spaces again has bounded geometry, since an $R$-ball in the wedge is contained in a finite union of $R$-balls in the factors.

The positive statement for $Y_b$ follows from the horocycle sequence $h_n=(n,1)$, $n\in\Z$, in the $\Hyp$ factor: by \cref{p:horocycle},
\[
        d_{Y_b}(h_n,h_m)=d_{\Hyp}(h_n,h_m)\asymp \ell(|n-m|).
\]
Restricting to $n\ge1$ gives the semi-infinite version.

It remains to exclude such a sequence in $X_b$.  Suppose $(q_n)$ were $\ell$-distorted in $X_b=T_D\vee E$.  By \cref{l:wedge-tail}, a tail lies entirely in $T_D$ or entirely in $E$ and, after reindexing, remains $\ell$-distorted in that factor.  It cannot lie in $T_D$, because \cref{c:tree-log} rules out logarithmic distortion in simplicial trees of arbitrary degree.  It cannot lie in $E=\R^2$, because of the Euclidean packing obstruction.  By condition~\textup{(iv)} of \cref{p:distorted-normalizations}, after passing to a fixed arithmetic subsequence, the first $N$ terms give $\asymp N$ uniformly separated points in a ball of radius $O(\log N)$.  Such a Euclidean ball contains only $O((\log N)^2)$ uniformly separated points, which would force
\[
        N\lesssim (\log N)^2,
\]
and is impossible.  Equivalently, \cref{l:poly-pack} with $Q=2$ would give $N^{1/2}\lesssim_\infty \ell(N)$.  Hence $X_b$ has no $\ell$-distorted semi-infinite sequence.

Since the restriction of $\ell$ to $\N$ is a rate function, \cref{l:qi-invariance} would transfer the $\ell$-distorted sequence in $Y_b$ to $X_b$ under any quasi-isometry, contradicting this conclusion.
\end{proof}

\begin{prop}[Large-scale comparison for the bounded-geometry pair]\label{p:bounded-standard-invariants}
Let $o$ denote the wedge point in each of the spaces $X_b$ and $Y_b$.
\begin{enumerate}[label=\textup{(\roman*)}]
\item Both $X_b$ and $Y_b$ are proper geodesic simply connected spaces of bounded geometry, and neither space is Gromov-hyperbolic.
\item For $Z\in\{X_b,Y_b\}$, put
\[
        V_Z^{\mathrm{unif}}(R)=\sup_{z\in Z}\Pack_1(B_Z(z,R)).
\]
Then
\[
        V_{X_b}^{\mathrm{unif}}(R)\asymp (D-1)^R
        \asymp V_{Y_b}^{\mathrm{unif}}(R).
\]
\item Their asymptotic dimensions satisfy
\[
        \asdim X_b=\asdim Y_b=2.
\]
\item Their Lipschitz filling functions satisfy
\[
        \delta_{X_b}\filleq L^2,
        \qquad
        \delta_{Y_b}\filleq L^2.
\]
\item For every sequence $(r_i)$ with $r_i\to\infty$ and every non-principal ultrafilter $\omega$,
\[
        \operatorname{Cone}_\omega(X_b,o,(r_i))
        \cong
        \operatorname{Cone}_\omega(Y_b,o,(r_i)).
\]
\end{enumerate}
\end{prop}

In particular, the cones in item~(v) have the same homeomorphism type and the same bi-Lipschitz type.

\begin{proof}
Properness, geodesicity, simple connectivity, and bounded geometry were proved in \cref{p:bounded-log-distinguishes}.  Both spaces contain an isometrically embedded copy of $\R^2$, so neither is Gromov-hyperbolic.

For item~(ii), use the notation from the statement.  In a finite one-point wedge, an $R$-ball is contained either in an $R$-ball of one factor or, if it reaches the wedge point, in the union of $R$-balls about the wedge point in the finitely many factors.  Hence $V_Z^{\mathrm{unif}}$ is equivalent to the maximum of the factors' uniform packing functions.

In the $D$-regular tree one has
\[
        V_{T_D}^{\mathrm{unif}}(R)\asymp (D-1)^R.
\]
The Euclidean factor contributes only polynomial growth, while the hyperbolic factor contributes $\asymp e^R$.  Since $D\ge 5$, we have $D-1>e$, and hence
\[
        e^R=o((D-1)^R).
\]
Hence
\[
        V_{X_b}^{\mathrm{unif}}(R)\asymp (D-1)^R\asymp V_{Y_b}^{\mathrm{unif}}(R),
\]
so uniform packing does not distinguish them.

For asymptotic dimension,
\[
        \asdim T_D=1,
        \qquad
        \asdim \R^2=2,
        \qquad
        \asdim \Hyp=2.
\]
The finite union theorem gives
\[
        \asdim X_b=\asdim Y_b=2.
\]

For the Lipschitz filling function of \cref{d:filling}, all factors are CAT$(0)$, and the Reshetnyak gluing theorem makes their one-point wedges CAT$(0)$ \cite[Chapter~II.11]{BH}.  The CAT$(0)$ quadratic isoperimetric inequality gives a quadratic upper bound; by Reshetnyak majorization, a loop of length $L$ admits a filling of area at most $L^2/(4\pi)$.

Conversely, collapsing the other factors to the wedge point gives a $1$-Lipschitz retraction onto the Euclidean plane, and composing a Lipschitz filling with this retraction cannot increase its parametrized area.  A Euclidean circle of radius $r$ has length $2\pi r$ and bounds a disk of area $\pi r^2=L^2/(4\pi)$.  Apply \eqref{e:lipschitz-stokes} to the standard primitive $\frac12(x\,dy-y\,dx)$ of the Euclidean area form and then use \eqref{e:area-form-bound}.  Every Lipschitz filling of the positively oriented circle therefore has area at least the enclosed disk area.  Thus the quadratic lower bound from $\R^2$ persists in each wedge. Hence
\[
        \delta_{X_b}\filleq L^2,
        \qquad
        \delta_{Y_b}\filleq L^2.
\]

Finally, fix the wedge point, a scaling sequence $(r_i)$, and a non-principal ultrafilter $\omega$.  By \cref{l:wedge-cones}, the cone is the wedge of the factor cones.  The asymptotic cone of the regular tree $T_D$ is a complete $\R$-tree.  To see homogeneity explicitly, let $[x_i]$ be a cone point and choose a vertex $v_i$ with $d_{T_D}(x_i,v_i)\le1/2$.  By vertex transitivity, choose a tree automorphism $g_i$ sending the fixed root to $v_i$.  The sequence $(g_i)$ induces a cone isometry sending the cone basepoint to $[v_i]=[x_i]$.

At the basepoint, each boundary ray $\rho$ of $T_D$ determines the cone ray
\[
        t\longmapsto
        \big[(\rho(\lfloor t r_i\rfloor))_i\big]_\omega,
        \qquad t\ge0.
\]
Two distinct boundary rays share only a finite initial segment, whose rescaled length tends to zero, so the corresponding cone rays lie in distinct components of the complement of the cone basepoint.  There are $2^{\aleph_0}$ such rays.  Since the cone itself has cardinality at most $2^{\aleph_0}$, the basepoint has valence exactly $2^{\aleph_0}$, and homogeneity gives the same valence at every point.  By uniqueness, this cone is the complete homogeneous $\R$-tree $\mathcal U$ of that valence.  Every asymptotic cone of $\Hyp$ is the same tree $\mathcal U$ \cite[Theorem~1.3.2]{DyubinaPolterovich}.  The Euclidean factor contributes $\R^2$.  Therefore
\[
        \operatorname{Cone}_\omega(X_b,o,(r_i))\cong \mathcal U\vee\R^2,
        \qquad
        \operatorname{Cone}_\omega(Y_b,o,(r_i))
        \cong \mathcal U\vee\R^2\vee\mathcal U.
\]
The wedge $\mathcal U\vee\mathcal U$ is a complete $\R$-tree, and every point still has valence $2^{\aleph_0}$.  By the uniqueness of the complete homogeneous $\R$-tree of this valence \cite[Theorem~1.1.3(ii)]{DyubinaPolterovich}, it is isometric to $\mathcal U$. By homogeneity, the isometry may be chosen to send the wedge point to the distinguished basepoint of $\mathcal U$.  Thus
\[
        \operatorname{Cone}_\omega(X_b,o,(r_i))
        \cong \operatorname{Cone}_\omega(Y_b,o,(r_i))
\]
for every choice of ultrafilter and scaling sequence, with the wedge point as basepoint.
\end{proof}

\begin{rem}[Why the bounded-geometry example uses logarithmic distortion]\label{r:bounded-log-not-sublog}
No sublogarithmic rate separates the bounded-geometry pair. By \cref{l:bg-exp-pack,l:exp-pack}, a bounded-geometry geodesic space admits no $o(\log N)$-distorted sequence.  The hyperbolic-plane factor supplies the sharp logarithmic rate through horocycles, while the Euclidean-plane factor is inserted only to equalize several standard invariants, including failure of hyperbolicity and quadratic filling behavior.
\end{rem}

\section{Summary and comparison tables}\label{s:summary}

The first table summarizes the power-rate classification. In the simplicial-tree row, ``has a ray'' and ``has a line'' refer to a geodesic ray and a bi-infinite geodesic line, respectively.

\begin{center}
\small
\setlength{\tabcolsep}{5pt}
\renewcommand{\arraystretch}{1.15}
\begin{tabular}{c|c|c}
Target space & semi-infinite powers $\N$ & bi-infinite powers $\Z$ \\
\hline
$\R$, Euclidean & $\alpha=1$ & $\alpha=1$ \\
$\R^k$, $k\ge2$ & $\alpha\in(1/k,1]$ & $\alpha\in(1/k,1]$ \\
Geodesic $\delta$-hyperbolic $X$ & \shortstack{$\alpha=1$ iff\\$\partial X\ne\varnothing$} & \shortstack{$\alpha=1$ iff\\$|\partial X|\ge2$} \\
$\Hyp$ & $\alpha=1$ & $\alpha=1$ \\
Simplicial tree $T$ & $\alpha=1$ iff $T$ has a ray & $\alpha=1$ iff $T$ has a line
\end{tabular}
\end{center}
\normalsize

In the hyperbolic-space row, $\partial X$ denotes the sequential Gromov boundary.  The second table summarizes the logarithmic and sublogarithmic results. Unless stated otherwise, positive entries apply to both index sets, while negative entries already rule out semi-infinite sequences.

\begin{center}
\small
\setlength{\tabcolsep}{5pt}
\renewcommand{\arraystretch}{1.15}
\begin{tabular}{c|c|c}
Target space & logarithmic distortion & sublogarithmic distortion \\
\hline
$\R^k$ & impossible & impossible \\
$\Hyp$ & possible by horocycles & impossible \\
Finitely generated groups & varies with the group & impossible for $o(\log N)$ \\
Simplicial trees & impossible & impossible for unbounded $f$ \\
Exponential-packing spaces & varies with the space & impossible \\
Arbitrary $\delta$-hyperbolic spaces & not determined by hyperbolicity alone & possible in general
\end{tabular}
\end{center}
\normalsize

The Euclidean impossibility in the first row follows from packing: $N$ separated points cannot fit inside a ball of radius $O(\log N)$ in $\R^k$, since such a ball contains only $O((\log N)^k)$ separated points. For finitely generated groups the sublogarithmic impossibility is \cref{c:fg-no-sublog}. For simplicial trees the stronger logarithmic impossibility is \cref{c:tree-log}. \Cref{s:qi-diagnostic} gives two diagnostic pairs. The first pair consists of proper geodesic $\delta$-hyperbolic spaces with unbounded geometry and the same basepoint coarse packing growth, asymptotic dimension, linear filling behavior, and asymptotic-cone type at the natural basepoint, but with $f_0(N)=\log\log(e^e+N)$ realized in one space and not the other. The second pair has bounded geometry: $X_b=T_D\vee\R^2$ and $Y_b=T_D\vee\R^2\vee\Hyp$ have the same uniform coarse packing growth, asymptotic dimension, quadratic filling behavior, asymptotic-cone type at the natural basepoint, and the same non-hyperbolic status, but logarithmic distortion is realized only in $Y_b$.

\section{Open problems}\label{s:open-problems}

We conclude with several problems concerning the sequence distortion spectra $\mathcal D_I(X)$ introduced in \cref{d:sequence-spectrum}.

\begin{prob}[Critical snowflake endpoints]\label{p:open-critical-endpoints}
Find general intrinsic conditions which rule out bi-Lipschitz embeddings
\[
        ([0,1], |s-t|^{1/Q})\longrightarrow X
\]
into a $Q$-dimensional model space $X$ at the critical exponent.
\end{prob}

In \cref{s:euclidean} the compactness argument produces a bounded-turning arc; Euclidean doubling then makes it a quasiarc, and quasisymmetric dimension distortion rules out full ambient dimension. A broader formulation should identify which combination of doubling, bounded turning, and ambient regularity yields the same endpoint obstruction in Ahlfors regular spaces.

\begin{prob}[Non-power rates in Euclidean spaces]\label{p:open-euclidean-rates}
Let $k\ge1$ and $I\in\{\N,\Z\}$. Determine $\mathcal D_I(\R^k)$ for rate functions beyond powers. In particular, if $f$ is a rate function satisfying
\[
        N=o(f(N)^k)
        \qquad\text{and}\qquad
        f(N)\preceq N,
\]
when does $[f]_{\simeq}\in\mathcal D_I(\R^k)$?
\end{prob}

The results settle power functions.  By \cref{l:poly-pack}, packing gives the necessary lower bound $N^{1/k}\lesssim_\infty f(N)$, while a polygonal path gives the trivial upper scale $N$. The positive examples for powers use self-similar chord-arc snowflakes. A complete answer for general $f$ would require either a flexible theory of quasiarcs with prescribed gauge or a proof that further regularity of $f$ is unavoidable.

\begin{prob}[Logarithmic distortion in hyperbolic groups]\label{p:open-hyperbolic-groups-log}
Characterize finitely generated word-hyperbolic groups whose Cayley graphs contain logarithmically distorted sequences.
\end{prob}

By \cref{l:qi-invariance}, the property is quasi-isometry invariant. Closed hyperbolic surface groups have it, because they are quasi-isometric to $\Hyp$ and horocycles in $\Hyp$ give logarithmic distortion. Free groups do not have it, because a Cayley graph with respect to a free basis is a tree and \cref{t:simplicial-trees} rules out all unbounded sublinear rates. A boundary-theoretic characterization might involve visual metrics, connectedness properties, quasiarcs in the boundary, or coarse horospherical directions.

\begin{prob}[Rates in negatively curved warped products]\label{p:open-warped-rates}
For which rate functions $f$ with $f(N)\preceq \log(1+N)$ does there exist a proper geodesic CAT$(-1)$ space, or a complete simply connected negatively curved surface, containing an $f$-distorted horocyclic sequence?
\end{prob}

\Cref{ex:loglog} shows that $f(N)=\log\log(e^e+N)$ can occur once bounded geometry is dropped. In a warped metric of the form
\[
        dt^2+e^{-2\varphi(t)}dx^2,
\]
the horocyclic distance scale is governed roughly by the inverse relation $\varphi(T)\asymp \log N$. Suitable choices of $\varphi$ may therefore realize many sublogarithmic functions. A classification should identify which choices are compatible with curvature bounds, completeness, properness, and prescribed local geometry.

\begin{prob}[Bounded-geometry hyperbolic diagnostic pairs]\label{p:open-bounded-hyp-diagnostic}
Determine whether there exist $\delta\ge0$, proper geodesic simply connected bounded-geometry $\delta$-hyperbolic spaces $X,Y$, and basepoints $o_X\in X$, $o_Y\in Y$ such that
\begin{enumerate}[label=\textup{(\roman*)}]
\item $\mathcal D_I(X)\ne\mathcal D_I(Y)$ for at least one $I\in\{\N,\Z\}$;
\item with
\[
        V_Z^{\mathrm{unif}}(R)=\sup_{z\in Z}\Pack_1(B_Z(z,R)),
        \qquad Z\in\{X,Y\},
\]
one has $V_X^{\mathrm{unif}}(R)\asymp V_Y^{\mathrm{unif}}(R)$;
\item $\asdim X=\asdim Y$;
\item $\delta_X\filleq\delta_Y$ for the Lipschitz filling functions;
\item for every sequence $(r_i)$ with $r_i\to\infty$ and every non-principal ultrafilter $\omega$,
\[
        \operatorname{Cone}_\omega(X,o_X,(r_i))
        \cong
        \operatorname{Cone}_\omega(Y,o_Y,(r_i)).
\]
\end{enumerate}
\end{prob}

The unbounded-geometry pair in \cref{p:seq-distinguishes,p:standard-invariants-agree} is itself hyperbolic and is separated by the log--log rate. The bounded-geometry pair in \cref{p:bounded-log-distinguishes,p:bounded-standard-invariants} is separated by the logarithmic rate, but it is made non-hyperbolic by the common $\R^2$ factor. A bounded-geometry example with both spaces hyperbolic would be sharper. Such an example would have to avoid the sublogarithmic obstruction of \cref{l:exp-pack}, so logarithmic distortion is the natural first test rate.

\begin{prob}[How much cone data determines sequence distortion?]\label{p:open-cones}
Let $X$ and $Y$ be proper geodesic spaces. Suppose that, for every sequence $(r_i)$ with $r_i\to\infty$, every non-principal ultrafilter $\omega$, and all basepoint sequences $(x_i)$ in $X$ and $(y_i)$ in $Y$, the cones
\[
        \operatorname{Cone}_\omega(X,(x_i),(r_i))
        \quad\text{and}\quad
        \operatorname{Cone}_\omega(Y,(y_i),(r_i))
\]
are bi-Lipschitz equivalent. Must
\[
        \mathcal D_I(X)=\mathcal D_I(Y)
\]
hold for each $I\in\{\N,\Z\}$?
\end{prob}

The examples in \cref{s:qi-diagnostic} compare cones only at the natural wedge basepoint, not all cones. Since the compactness arguments in this note pass from a distorted sequence to a snowflaked interval inside an asymptotic cone, sufficiently uniform cone data should constrain sequence distortion, but sequence distortion is also a global realization problem.

\begin{prob}[One-sided versus two-sided spectra]\label{p:open-one-two}
\leavevmode
\begin{enumerate}[label=\textup{(\roman*)}]
\item For which large-scale classes of spaces does
\[
        \mathcal D_{\N}(X)=\mathcal D_{\Z}(X)?
\]
\item Under what hypotheses on $X$ does the following implication hold: if $f$ is a rate function and $X$ contains an $f$-distorted sequence indexed by $\N$, then $X$ contains a $g$-distorted sequence indexed by $\Z$ for some rate function $g\simeq f$?
\end{enumerate}
\end{prob}

The Euclidean and hyperbolic-plane constructions are bi-infinite; restriction gives semi-infinite examples.  At the linear scale in a geodesic hyperbolic space, \cref{p:hyperbolic-boundary-linear} identifies the distinction exactly: a semi-infinite sequence exists if and only if the sequential boundary is nonempty, whereas a bi-infinite sequence exists if and only if the boundary contains at least two points.  Trees with a ray but no bi-infinite line give the simplest examples where these conditions differ.  Beyond the hyperbolic setting, the general obstruction should be related to ends, bottlenecks, and the existence of two-sided coarse rays in compatible directions.

\begin{prob}[Sequence distortion in finitely generated groups]\label{p:open-groups}
For each $I\in\{\N,\Z\}$, describe $\mathcal D_I(G)$ for finitely generated groups $G$ with word metrics, especially within standard quasi-isometry classes such as nilpotent groups, relatively hyperbolic groups, mapping class groups, and right-angled Artin groups.
\end{prob}

The results give several boundary cases: finitely generated groups never realize nondecreasing $o(\log N)$ functions by \cref{c:fg-no-sublog}; word-hyperbolic groups realize no powers $N^\alpha$ with $\alpha<1$ by \cref{t:delta-power}; and the discrete Heisenberg group realizes the square-root rate in its center. A systematic theory would relate $\mathcal D_I(G)$ to subgroup distortion, divergence, Morse directions, peripheral subgroups, and the geometry of asymptotic cones.

\begin{prob}[Stability under coarse constructions]\label{p:open-coarse-constructions}
For each $I\in\{\N,\Z\}$, determine how $\mathcal D_I(X)$ behaves under standard large-scale constructions, such as products, trees of spaces, relatively hyperbolic cusped spaces, warped cones, and wreath products?
\end{prob}

The finite wedge lemma in \cref{s:qi-diagnostic} gives one simple case: an unbounded-rate distorted sequence in a finite wedge eventually lies in one factor. Products and trees of spaces should be subtler, because a sequence may distribute its motion among several directions. Understanding these operations could sharpen sequence distortion as a quasi-isometry invariant and yield examples beyond the model geometries considered here.

\section{Disclosure of AI use}

ChatGPT was used in preparing this manuscript. The author independently checked and edited all mathematical arguments and takes full responsibility for the final content.

\end{document}